\documentclass{article}
\usepackage{amsfonts, amsthm}
\usepackage{amsmath}
\usepackage{amssymb}
\usepackage{graphicx}

\newtheorem{theorem}{Theorem}[section]
\newtheorem{cor}{Corollary}[section]
\newtheorem{lemma}{Lemma}[section]
\newtheorem{proposition}{Proposition}[section]
\theoremstyle{definition}
\newtheorem{definition}{Definition}[section]
\newtheorem{example}{Example}[section]

\newtheorem{question}{Question}[section]
\numberwithin{equation}{section}

\date{ }

\begin{document}

\title{Self-Similar Polygonal Tiling}
\author{Michael Barnsley \\ Australian National University \\ Canberra, Australia \\ \\
Andrew Vince \\ University of Florida \\ Gainesville, FL, USA}

\maketitle

\begin{abstract}
The purpose of this paper is to give the flavor of the subject of self-similar
tilings in a relatively elementary setting, and to provide a novel method for
the construction of such polygonal tilings.
\end{abstract}

\section{Introduction.}  \label{sec:intro}  Our goal is to lure the reader into the theory 
underlying the figures scattered throughout this paper.  The individual polygonal tiles in each of these tilings are pairwise similar, and there are only finitely many up to congruence.  
Each tiling is {\it self-similar}.  None of the tilings are periodic, yet 
each is {\it quasiperiodic}.  These concepts, self-similarity and quasiperiodicity, are defined 
in Section~\ref{sec:polytilings} and are discussed throughout the paper.  
Each tiling  is constructed by the same method from a 
single self-similar polygon.

\begin{figure}[tbh]
\vskip -4mm
\includegraphics[width=13cm, keepaspectratio]{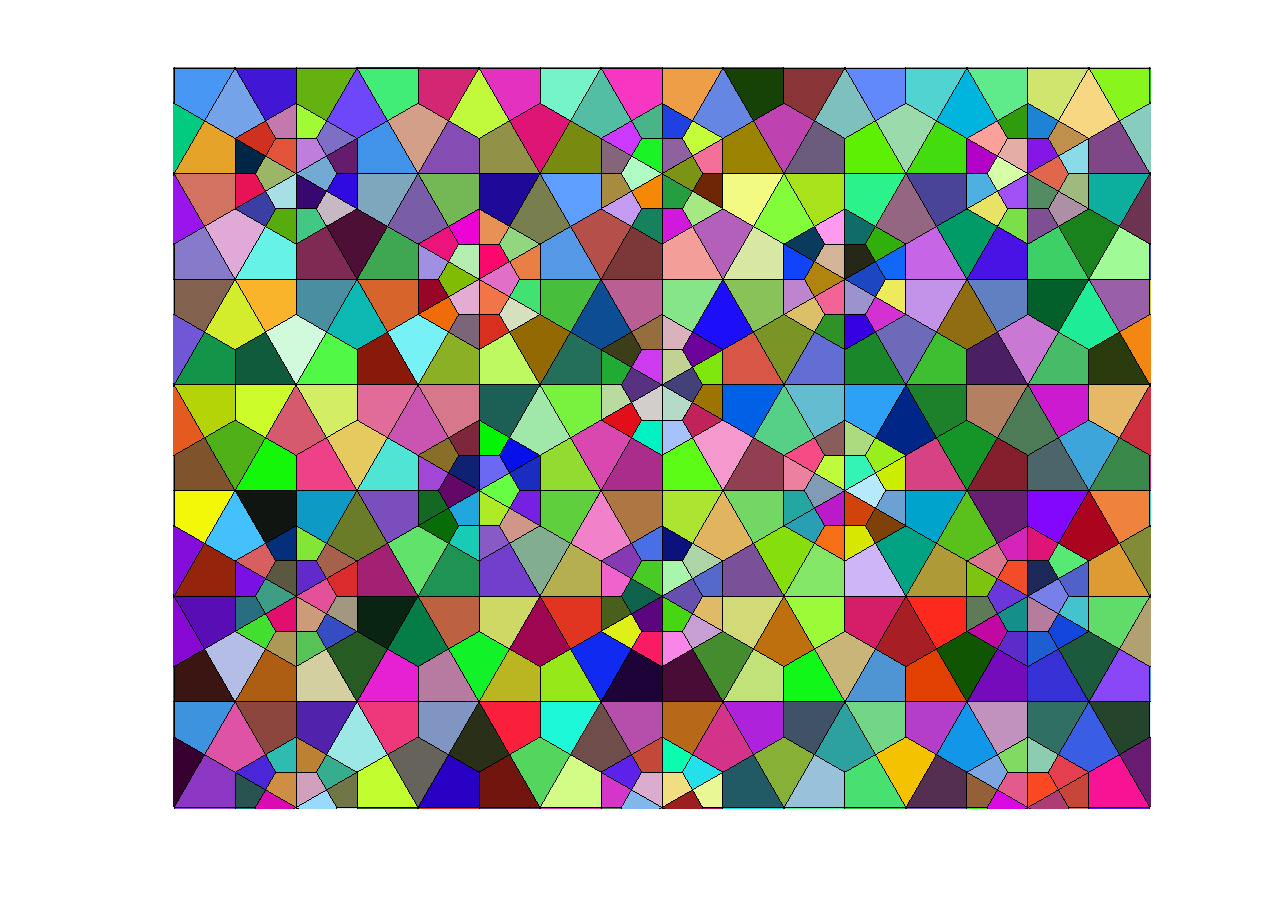} 
\vskip -6mm
\caption{A self-similar polygonal tiling of order $2$.}
\label{fig:5}
\end{figure}

For the tiling $T$ of the plane, a part of which is shown in
Figure~\ref{fig:5}, there are  two similar tile shapes, the ratio of
the sides of the larger quadrilateral to the smaller quadrilateral being
$\sqrt{3}:1$. In the tiling of the entire plane, the part shown in the figure
appears \textquotedblleft everywhere," the phenomenon known as quasiperiodicity
or repetitivity.  The tiling is self-similar in that there exists
 a similarity transformation $\phi$ of the plane such that, for each tile
$t\in T$, the \textquotedblleft blown up" tile $\phi(t)=\{\phi(x):x\in t\}$ is
the disjoint union of the original tiles in $T$. 

\begin{figure}[tbh]
\centering
\includegraphics[width=11cm, keepaspectratio]{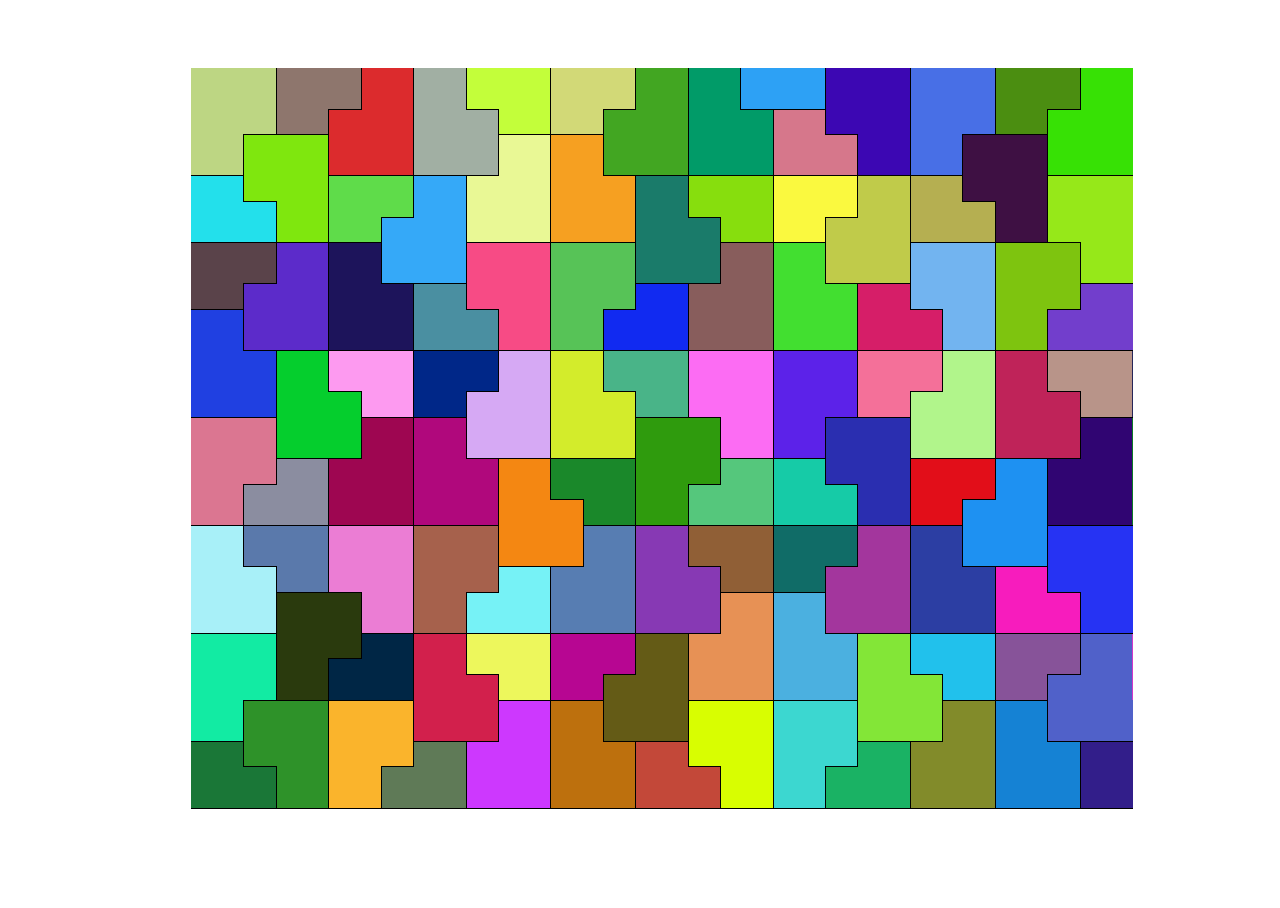}
\vskip -4mm
\caption{A golden bee tiling.}
\label{fig:2}
\end{figure}

The tiling in Figure~\ref{fig:2} is also self-similar and quasiperiodic. 
There are again two tile shapes, shown in dark and light grey in 
Figure~\ref{figure1}.  The large (dark grey) tile (L), the small (light grey) tile (S), 
and their union, call it $G$, are pairwise similar polygons.  
The hexagon $G$, called the golden bee in \cite{S}, appears in 
\cite{G} where it is attributed to Robert Ammann. For an entertaining piece
on ``The Mysterious Mr. Ammann," see the article by M.
Senechal \cite{Sen}.  If $\tau=(1+\sqrt{5})/2$ is the golden ratio
and $a=\sqrt{\tau}$, then the sides of L are $a$ times as long as the sides
of S, and the sides of $G$ are $a$ times as long as the sides of L.
Except for non-isosceles right triangles, the golden bee is the only polygon that can be partitioned into a 
non-congruent pair of scaled copies of itself \cite{Sch}.

\begin{figure}[bht]
\centering
\includegraphics[width=3.5cm, keepaspectratio]{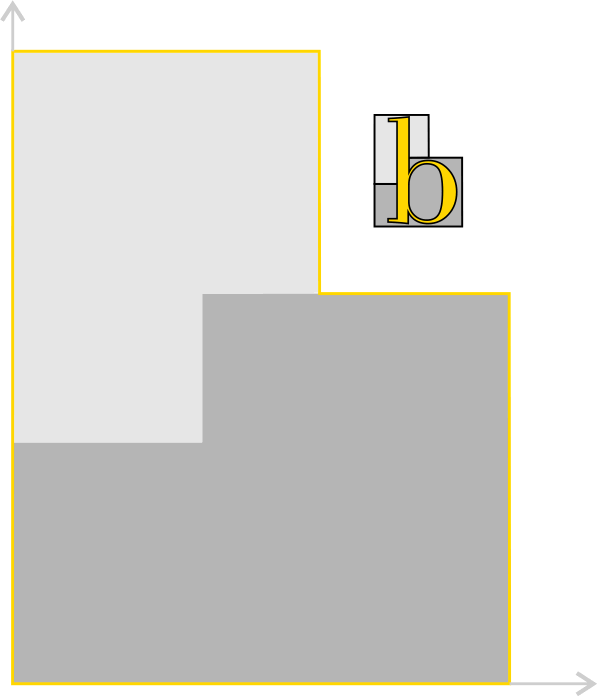}
\vskip 4mm
\caption{The golden bee is an example of a self-similar
polygon.  The ratio of the length of the left side to the
rightmost side is the golden ratio, and also of the bottom side to the topmost side. The inset
picture relates this hexagon to the letter bee.}
\label{figure1}
\end{figure}

\begin{figure}[htb]
\centering
\includegraphics[width=10cm, keepaspectratio]{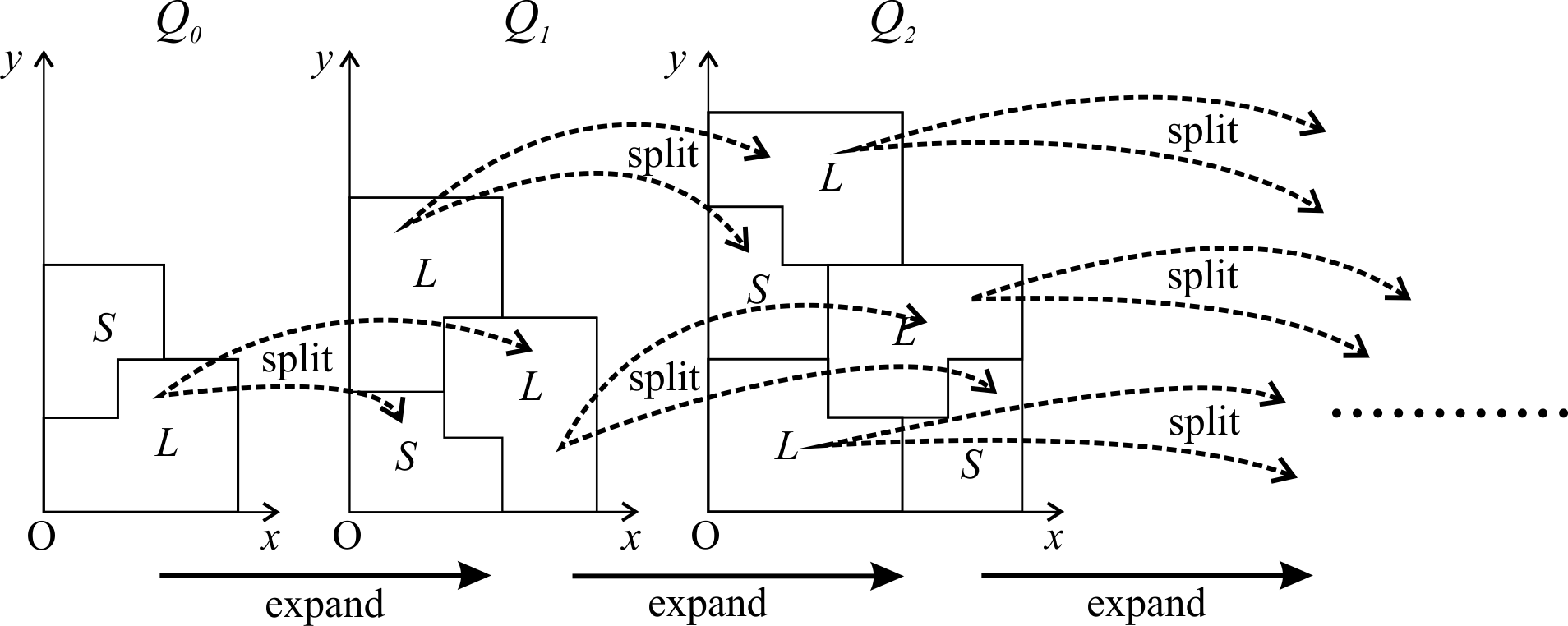}
\vskip 3mm
\caption{Inductive construction of the canonical sequence of
golden tilings $\{Q_{k}\}$. Each tiling is obtained from its predecessor by
scaling by the square root of the golden ratio, then subdividing each inflated large tile
into a large (L) and small (S) tile, as in Figure
\ref{figure1}.}
\label{figure2}
\end{figure}

Before veering into the general case, we give an
informal description of how infinitely many golden bee tilings, like the one in Figure~\ref{fig:2}, can be
obtained from the golden bee polygon in Figure~\ref{figure1}.
Figure \ref{figure2} illustrates how a canonical sequence of tilings $\{Q_{k}\}_{k=0}^{\infty}$ can be
constructed recursively by ``expanding and splitting." Each
tiling in the sequence uses only copies of the large (L) dark grey and small (S) light grey tiles of Figure \ref{figure1}.
Note that $Q_2$ is the disjoint union of one copy of $Q_0$ and one copy of $Q_1$.  Similarly, $Q_{k+2}$ is the disjoint union of a copy of $Q_{k}$ and a copy of $Q_{k+1}$ for all $k\geq 0$.  
This provides a way of embedding a copy of $Q_k$ into a copy of $Q_{k+1}$ (call this
a type $1$ embedding), and another way of embedding a copy of $Q_k$ into a copy of $Q_{k+2}$ (call this a type $2$ embedding).  The first type applied successively twice yields a
different tiling from the one obtained by applying the second type once.

\begin{figure}[htb]
\centering
\includegraphics[width=12cm, keepaspectratio]{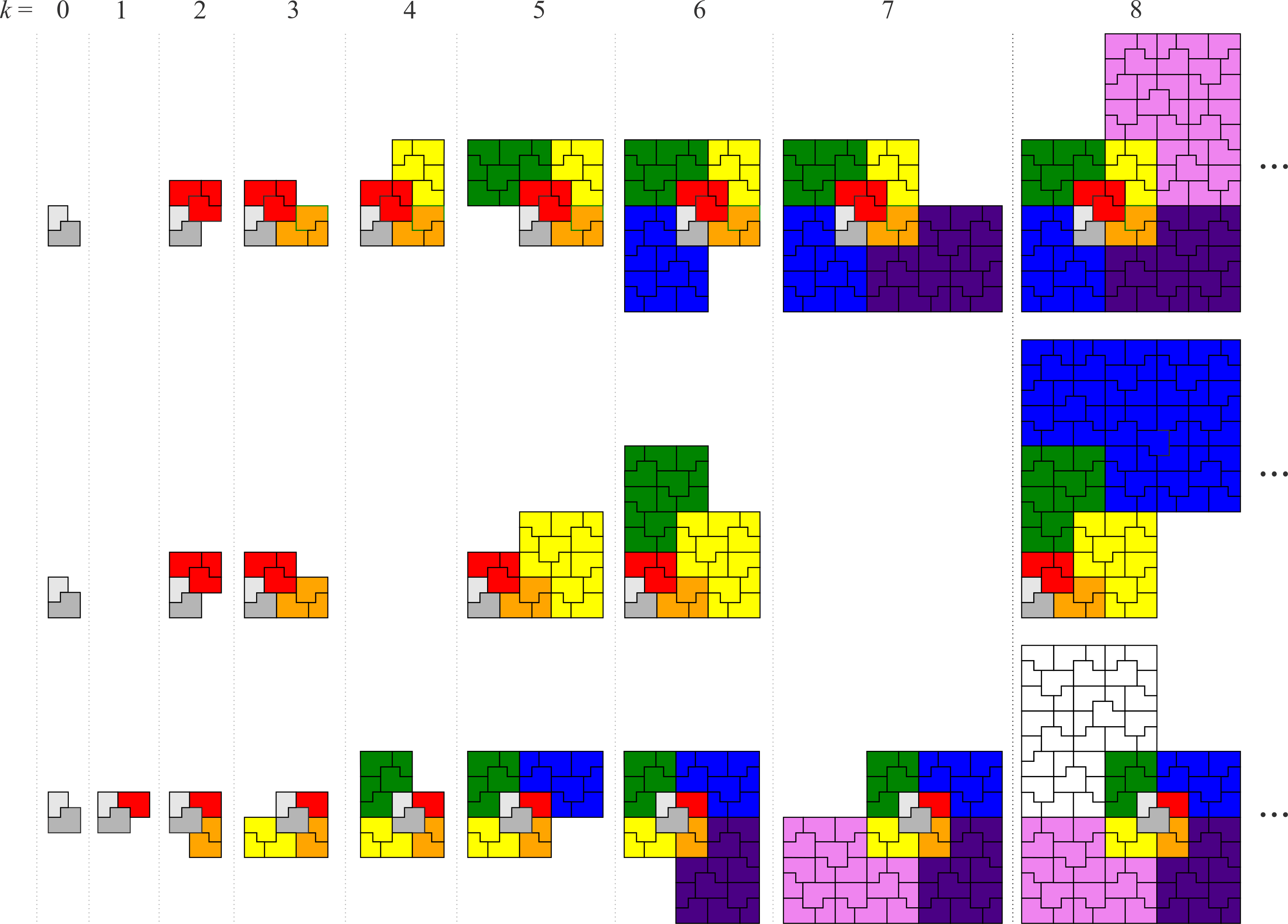}
\vskip 4mm
\caption{The construction of $G(11111111)$ (bottom),
$G(21212)$ (middle) and $G(21111 1 1)$ (top).  The tilings in the $k$th column are congruent to 
$Q_{k}$.  In general, the tiling $G(\theta_{1}\theta
_{2} \dots \theta_{K})$ is congruent to $Q_{\theta_{1}+\theta_{2}+\dots+\theta_{K}}$.}
\label{figure3}
\end{figure}

Fix a copy of $Q_0$ (the first column of Figure~\ref{figure3}).  It follows from the paragraph above that, for each infinite sequence of the symbols $1$ and $2$, for example $21212 \cdots$, one obtains a nested sequence of tilings, each tiling in the sequence congruent to $Q_k$ for some $k$.  The example
$G(2)\subset G(21)\subset G(212)\subset G(2121)\subset G(21212) \subset \cdots$ is illustrated in the middle row in Figure \ref{figure3}.  
Similarly, the top row illustrates the construction of $G(2111111)$ and the
bottom row illustrates the construction of $G(11111111)$.  
The union of the tiles, for example
\[
G(21212 \dots )=G(2)\cup G(21)\cup G(2121)\cup  G(21212) \cup \cdots ,
\]
is a tiling of a region in the plane of infinite area.  In this way, for each infinite sequence
$\theta$ with terms in $\{1,2\}$,
there is a corresponding tiling $G(\theta)$, which is referred to as a {\it golden bee tiling}. 
This ad hoc procedure for obtaining golden bee tilings from the single golden bee $G$ 
has a simple description in the general case.  This is the construction 
in Definition~\ref{def:scaled} of Section~\ref{sec:GP}.

Properties of the golden bee tilings include the following, extensions to
the general case appearing in Section~\ref{sec:ss}.

\begin{itemize}
\item $G(\theta)$ is a tiling of the plane for almost all $\theta$. What is
meant by ``almost all" and for which $\theta$ the statement is true is
discussed in Section~\ref{sec:ss}.

\item $G(\theta)$ is self-similar and quasiperiodic for
infinitely many $\theta$. Results on precisely which $\theta$ appear in
Section~\ref{sec:ss}.

\item $G(\theta)$ is nonperiodic for all $\theta$.

\item There are uncountably many distinct golden bee tilings up to congruence.

\item The ratio of large to small tiles in any a ball of radius $R$ centered
at some fixed point tends to the golden ratio $(1+\sqrt{5})/2$ as
$R\rightarrow\infty$. The general method for calculating such ratios is
demonstrated after Example~\ref{ex:trap} in Section~\ref{sec:examples}.
\end{itemize}

\begin{figure}[hbt]
\vskip -5mm 
\centering
\includegraphics[width=13cm, keepaspectratio]{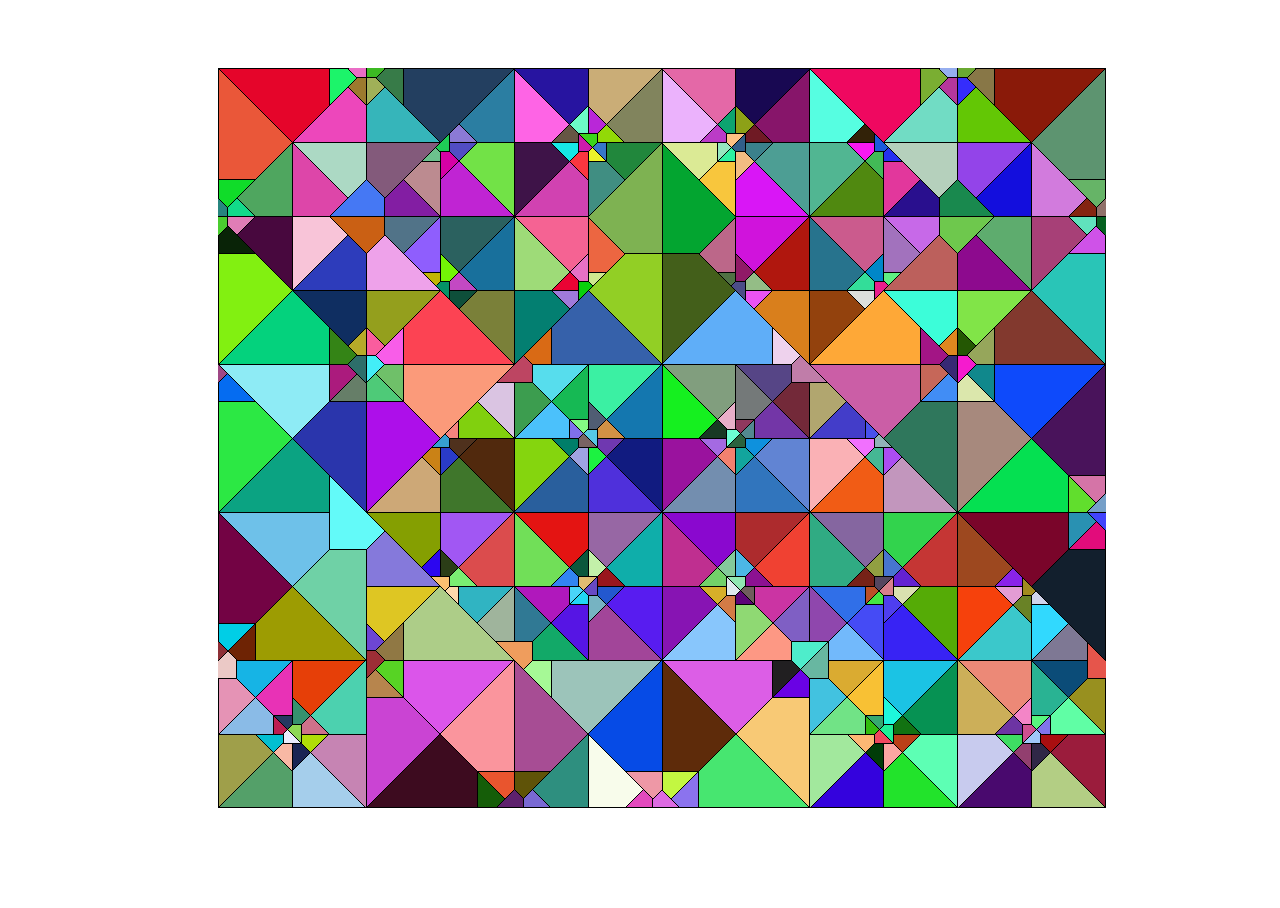} 
\vskip -5mm
\caption{A self-similar polygonal tiling of order $6$;
see Example~\ref{ex:spor}.}
\label{fig:3}
\end{figure}

\section{What is in this paper.}

\label{sec:contents} The organization of this paper is as follows.  As a  motivating example, we informally explored
the golden bee tilings in Section~\ref{sec:intro}.  Section \ref{sec:polytilings} contains background and definitions, in particular an explanation of exactly what is meant by a self-similar polygonal tiling.
Our general construction of self-similar polygonal tilings, the subject of Section~\ref{sec:GP}, 
 is based on what we call a generating pair (Definition~\ref{def:GP}).  The
crux of the construction, generalizing that of the golden tilings of Section \ref{sec:intro}, is contained in Definition~\ref{def:scaled}.  The main
result of the paper is Theorem~\ref{thm:main0}, stating that our construction
indeed yields self-similar polygonal tilings of the plane.
Examples of self-similar polygonal tilings appear in Section~\ref{sec:examples}; the question of which polygons admit self-similar
tilings leads to an intriguing polygonal taxonomy.  Section~\ref{sec:ss}
elaborates on Theorem~\ref{thm:main0}, delving into some of its subtleties.  Proofs of results in Section~\ref{sec:ss}
appear in Section~\ref{sec:proofs}. 
There remains much to be learned about self-similar polygonal tiling; basic
problems, posed in Section~\ref{sec:remarks}, remain open.
\vskip -4mm

\begin{figure}[htb]
\centering
\includegraphics[width=10cm, keepaspectratio]{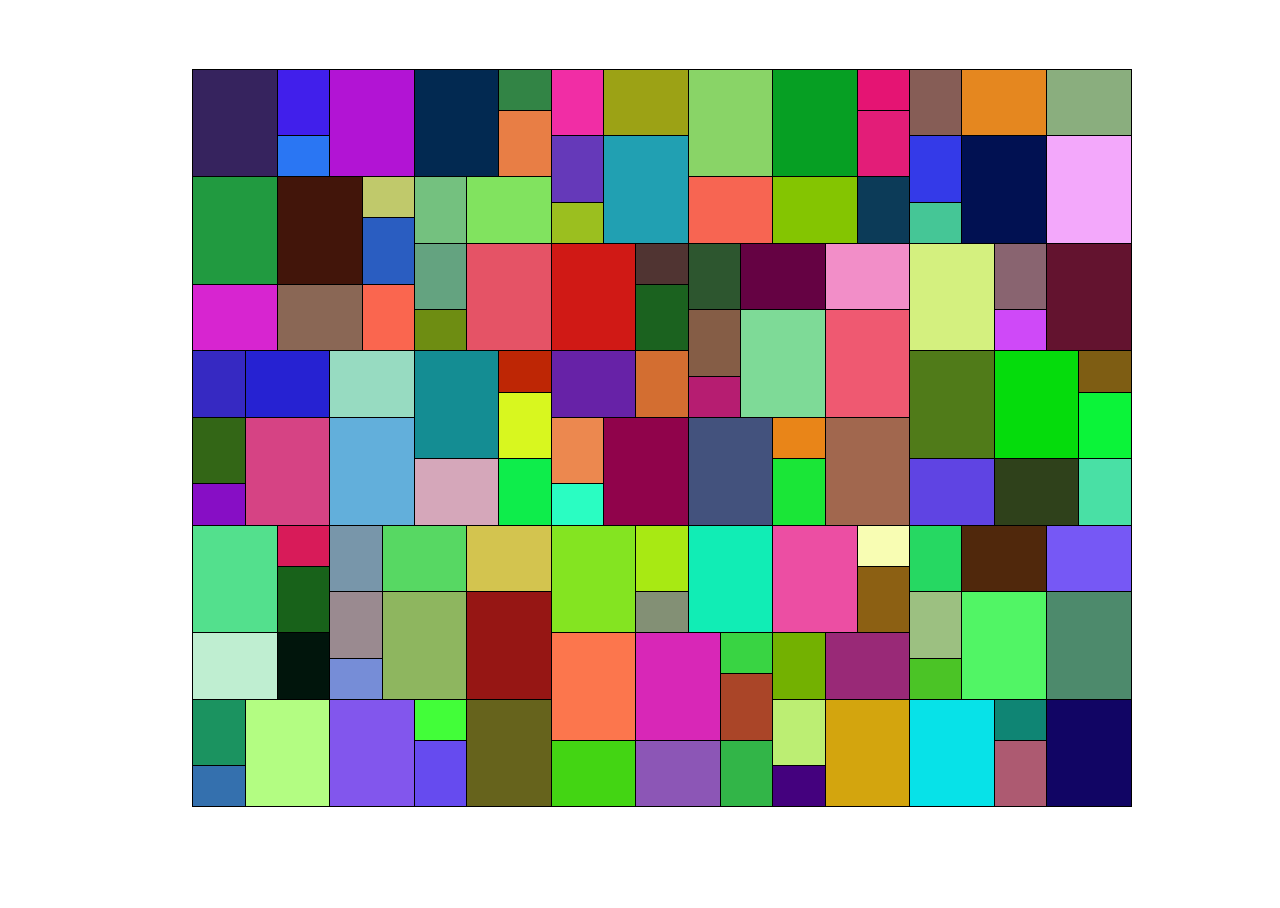} \vskip -4mm
\caption{A
tiling based on sporadic generating pair A in Figure~\ref{fig:spor}; see
Example~\ref{ex:spor}.}
\label{fig:A}
\end{figure}

\section{Self-Similar Polygonal Tilings.} \label{sec:polytilings} 
There is a cornucopia of tilings of the plane
possessing some sort of regularity. The history of such tilings goes back to
antiquity, and the mathematical literature is replete with papers on the
subject. On the decorative side, there are, for example, the $14^{th}$ century
mosaic tilings on the Alhambra palace in Granada, Spain, and the tilings in M.
C. Escher's $20^{th}$ century prints. On the mathematical side, there are, for
example, the tilings by regular polygons dating back at least to J. Kepler,
tilings with large symmetry group as studied by Gr\"{u}nbaum and Shephard
\cite{G} and many others, and the Penrose tilings \cite{Pen} and their relatives.
Our goal in this paper is to give the flavor of the subject of self-similar tilings in a
relatively elementary setting, and to provide a novel method for the construction of
such polygonal tilings.

After a few basic definitions, we take a closer look at the concepts of
self-similarity and quasiperiodicity. A \textit{tiling} of the plane is a set
of pairwise disjoint compact sets whose union is the plane. \textit{Disjoint}
means that a pair of tiles can intersect only at a subset of their
boundaries. A set $\mathcal{P}$ of tiles is called the \textit{prototile set}
of a tiling $T$ if, up to congruence, $\mathcal{P}$ contains exactly one copy
of each tile in $T$. The prototile set in Figure~\ref{fig:5} consists of the
two quadrilaterals. The \textit{order} of the tiling is the number of tiles in
its prototile set.  Whereas the tilings in Figures~\ref{fig:5} and \ref{fig:2} have order $2$,
the tiling in Figure~\ref{fig:3} has order $6$.  
 In this paper, all tilings $T$ have finite order, and all
the tiles in $T$ have the same shape up to similarity. Moreover, all tilings
in this paper are polygonal, the tiles being closed polygons with non-crossing
(except at vertices) sides and positive area. Therefore, the prototile set of
our tilings will consist of finitely many pairwise similar polygons.

A tiling of the plane is \textit{periodic} if there is a translation of the
plane that leaves the tiling invariant (as a whole fixed); otherwise the
tiling is \textit{nonperiodic}. Quasiperiodicity, a property less stringent
than periodicity, has gained considerable attention since the 1984 Nobel Prize
winning discovery of quasicrystals by Shechtman, Blech, Gratias, and Cahn
\cite{SBGC}. Quasicrystals are materials intermediate between crystalline and
amorphous, materials whose molecular structure is nonperiodic but nevertheless
exhibits long range order as evidenced by sharp ``Bragg" peaks in their
diffraction pattern. Define a \textit{patch} of a tiling $T$ as a subset of
$T$ whose union is a topological disk. A tiling of the plane is
\textit{quasiperiodic} if, for any patch $U$, there is a number $R>0$ such
that any disk of radius $R$ contains, up to congruence, a copy of $U$. This is
what we meant by saying that every patch of the tiling appears everywhere in
the tiling. If you were placed on a quasiperiodic tiling, then your local
surroundings would give no clue as to where you were globally. The tilings in
Figures~\ref{fig:5}, \ref{fig:2}, and \ref{fig:3}, although nonperiodic, are
quasiperiodic.

A \textit{similitude} $f$ of the plane is a transformation with the property
that there is a positive real number $r$, the \textit{scaling ratio}, such
that $|f(x) - f(y)| = r\, |x-y|$, where $| \, \cdot\, |$ is the Euclidean
norm. A similitude of the plane is, according to a classification, either a
stretch rotation or a stretch reflection.
Self-similarity, in one form or another, has been intensely studied over the
past few decades --- arising in the areas of fractal geometry, Markov
partitions, symbolic dynamics, radix representation, and wavelets. The tiles
arising in these subjects usually have fractal boundaries. R. Kenyon
\cite{Ken}, motivated by work of W. P. Thurston \cite{Thurs}, proved the
 existence of a wide class of self-similar
tilings. Explicit methods exist for the construction of certain families of
self-similar tilings: digit and crystallographic tilings \cite{V1} and the
Rauzy \cite{Rauzy} and related tilings. In this paper, a tiling $T$ is called
\textit{self-similar} if there is a similitude $\phi$ with scaling ratio
$r(\phi) > 1$ such that, for every $t\in T$, the polygon $\phi(t)$ is the
disjoint union of tiles in $T$. The map $\phi$ is called a
\textit{self-similarity} of the tiling $T$.

Since all tiling figures in this paper are, of necessity, just
a part of the tiling, and because quasiperiodicity and self-similarity are
global properties, it is not possible to say, from the figure alone,
whether or not the actual tiling is quasiperiodic or self-similar.

In order to keep technicalities to a minimum, we restrict this paper to polygonal tiling.

\begin{definition}
A \textbf{self-similar polygonal tiling} is a tiling of the plane by pairwise
similar polygons that is (1) self-similar, (2) quasiperiodic, and (3) of
finite order.
\end{definition}

Immediate consequences of the above definition are the following.

\begin{itemize}
\item Self-similar polygonal tilings are hierarchical. Using the notation
$\phi^{k}$ for the $k$-fold composition, if $T$ is a self-similar tiling with
self-similarity $\phi$, then \linebreak$T, \phi(T), \phi^{2}(T), \dots$ is a
sequence of nested self-similar tilings, each at a larger scale than the
previous. \vskip 2mm

\item If $p$ is any polygon in the prototile set of a self-similar polygonal
tiling, then there exist similitudes $f_{1},f_{2},\dots,f_{N},\,N\geq2,$ each
with scaling ratio less than $1$, such that
\begin{equation}
p=\bigsqcup_{n=1}^{N}f_{n}(p)\text{,}\label{eq:IFS}
\end{equation}
where $\bigsqcup$ denotes a pairwise disjoint union. In the fractal
literature, \linebreak$F=\{f_{1},f_{2},\dots,f_{N}\}$ is an example of an
iterated function system and $p$ is an example of the attractor of the iterated
function system \cite{FE}.
\end{itemize}

\section{The construction.} \label{sec:GP}

Our construction of self-similar polygonal tilings is contained in
Definition~\ref{def:scaled}. It relies on what we call a generating pair,
whose Definition~\ref{def:GP} is clearly motivated by the
equation~\eqref{eq:IFS} that must hold for any self-similar polygon tiling.

\begin{definition}
\label{def:GP} Let $p$ be a polygon and $F=\{f_{1},f_{2},\dots,f_{N}\}, \,
N\geq2,$ a set of similitudes with respective scaling ratios $r_{1} ,r_{2}, \dots ,r_{N}$. If there there exists a real number $0<s<1$ and positive
integers $a_{1},a_{2}, \dots, a_{N}$ such that
\begin{equation} \label{eq:I}
p=\bigsqcup_{f\in F} f(p), \qquad\qquad\text{and} \qquad\qquad
r_{n}=s^{a_{n}},
\end{equation}
for $n=1,2,\dots, N$, then $(p,F)$ will be called a \textbf{generating pair}.
The second equation is essential for insuring that the constructed tilings
have finite order; see Question 3 and Proposition~\ref{prop:tiling} of
Section~\ref{sec:ss}.
\end{definition}

\begin{example}
[Generating pair for the golden bee]Let $s=1/\sqrt{\tau}$, where $\tau$ is the
golden ratio. For the golden bee, the generating pair is $(G,\{f_{1},f_{2}\})$, where
\[
f_{1}
\begin{pmatrix}
x\\
y
\end{pmatrix}
=
\begin{pmatrix}
0 & -s\\
s & 0
\end{pmatrix}
\begin{pmatrix}
x\\
y
\end{pmatrix}
+
\begin{pmatrix}
s\\
0
\end{pmatrix}
,\qquad\qquad f_{2}
\begin{pmatrix}
x\\ y
\end{pmatrix}
=
\begin{pmatrix}
s^{2} & 0\\0 & -s^{2}
\end{pmatrix}
\begin{pmatrix}
x\\y
\end{pmatrix}
+
\begin{pmatrix}
0\\1
\end{pmatrix}.
\]
The respective scaling ratios are $r_{1}=s,\,r_{2}=s^{2}$ and $(a_{1}
,a_{2})=(1,2)$.
\end{example}

 \begin{figure}[hbt]
\vskip -5mm
\includegraphics[width=13cm, keepaspectratio]{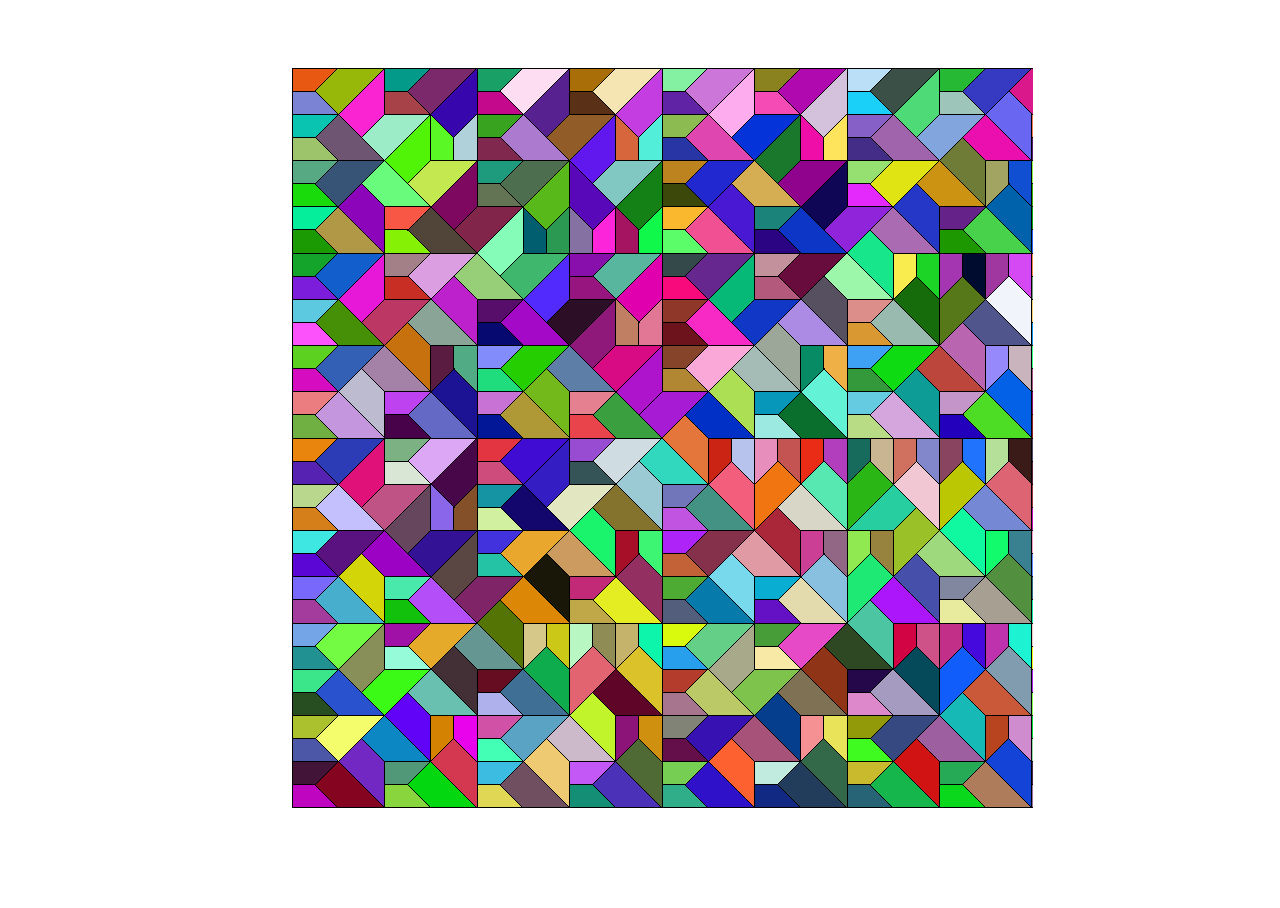}
\vskip -6mm
\caption{Tiling based on sporadic generating pairs C in
Figure~\ref{fig:spor}; see Example~\ref{ex:spor}.}
\label{fig:8}
\end{figure}

From the equations~\eqref{eq:I}, since the area of $p$ is equal to the sum of
the areas of $f(p)$, for $f\in F$, we must have
\begin{equation}
\label{eq:ratio}\sum\limits_{n=1}^{N} \, s^{2a_{n}}=1\text{.}
\end{equation}
Note that, for any set $\{ a_{1},a_{2}, \dots ,a_{N}\}$ of positive integers,
equation~\eqref{eq:ratio} has a unique positive solution $s$. We will, without
loss of generality, always assume that $g :=\gcd\,(a_{1},\dots,a_{N})=1$;
otherwise $s$ can be replaced by $s^{g}$.

Let $[N] = \{1,2,3, \dots,N \}$. Denote by $[N]^{*}$ the set of all finite
strings over the alphabet $[N]$ and by $[N]^{\omega}$ the set of all infinite
strings over the alphabet $[N]$. For $\sigma\in[N]^{*}$, the length of
$\sigma$ is denoted $|\sigma|$. The following simplifying notation is useful.
For $\sigma= \sigma_{1}\, \sigma_{2}\, \dots\, \sigma_{k} \in[N]^{*}$ let

\[
e(\sigma) := \sum_{i=1}^{|\sigma|} a_{\sigma_{i}} \qquad\qquad\qquad
e^{-}(\sigma) := \sum_{i=1}^{|\sigma|-1} a_{\sigma_{i}}
\]
\[
\begin{aligned} f_{\sigma} &:= f_{\sigma_1}\circ f_{\sigma_2 } \circ \cdots
\circ f_{\sigma_k} \\ f_{-\sigma} &:= f_{\sigma_1}^{-1}\circ f_{\sigma_2
}^{-1} \circ \cdots \circ f_{\sigma_k}^{-1}. \end{aligned}
\]
For $\theta\in[N]^{\omega}$ let
\[
\theta|k := \theta_{1} \, \theta_{2}\, \dots\theta_{k}.
\]

From a single generating pair $(p,F)$, a potentially infinite number of
self-similar polygonal tilings will be constructed. There are three steps in
the construction. All tiling figures in this paper are based
on the construction in Definition~\ref{def:scaled}.

\begin{definition}
\label{def:scaled} Let the generating pair $(p,F)$ and $\theta\in[N]^{\omega}$
be fixed.
\vskip 2mm

(1) For each positive integer $k$ and each $\sigma\in\lbrack N]^{\ast}$,
construct a single tile $t(\theta,k,\sigma)$ that is similar to $p$:
\[
t(\theta,k,\sigma):=(f_{-(\theta\,|\,k)}\circ f_{\sigma})(p).
\]
\vskip2mm (2) Form a patch $T(\theta,k)$ consisting of all tiles
$t(\theta,k,\sigma)$ for which $\sigma$ satisfies a certain property:

\[
T(\theta,k):=\left\{  t(\theta,k,\sigma)\,:\,e(\sigma) \geq e(\theta|k) >
e^{-}(\sigma)\right\}  .
\]
\vskip2mm (3) The final tiling ${T}(\theta)$, depending only on $\theta$, is
the nested union of the patches $T(\theta,k)$:

\[
{T}(\theta):=T(p,F,\theta):=\bigcup_{k\geq1}T(\theta,k).
\]
The tiling $T(\theta)$ is called a $\theta$-\textbf{tiling generated by the
pair} $(p,F)$. Each set $t(\theta,k,\sigma)\in T(\theta)$ is a tile of
$T(\theta)$.
\end{definition}

\vskip 2mm

The patches $T(\theta,k)$ are nested because every tile in $T(\theta,k)$ is a
tile in $T(\theta, k+1)$:
\[
f_{-(\theta|{k})}\circ f_{\sigma}(p) =f_{-(\theta|{k})}\circ(f_{\theta_{k+1}
})^{-1}\circ f_{\theta_{k+1}}\circ f_{\sigma}(p)=f_{-(\theta| {k+1})} \circ
f_{\theta_{k+1}\sigma}(p).
\]
Figure~\ref{fig:tree} illustrates the tree-like structure underlying the
construction of $T(\theta,3)$, where the generating pair is, for example, the
golden bee of Figure~\ref{figure1} and $\theta= 1\, 2 \,1 \cdots$. The eight
tiles in $T(\theta,3)$ are represented by the leaves of the tree, three of
these tiles (in black) are small and five (in red) are large (larger by a
factor of $\sqrt{\tau}$). The numbers on the edges are $a_{1} =1 $ and $a_{2}
=2$. Each sequence $\sigma$ of edge labels from the root to a leaf satisfies
$e(\sigma) \geq e(\theta|3) = 4 > e^{-}(\sigma)$. The numbers on the leaves
are $e(\sigma)$ for the corresponding sequence.

\begin{figure}[tbh]
\vskip -4mm 
\centering
\includegraphics[width=11cm, keepaspectratio]{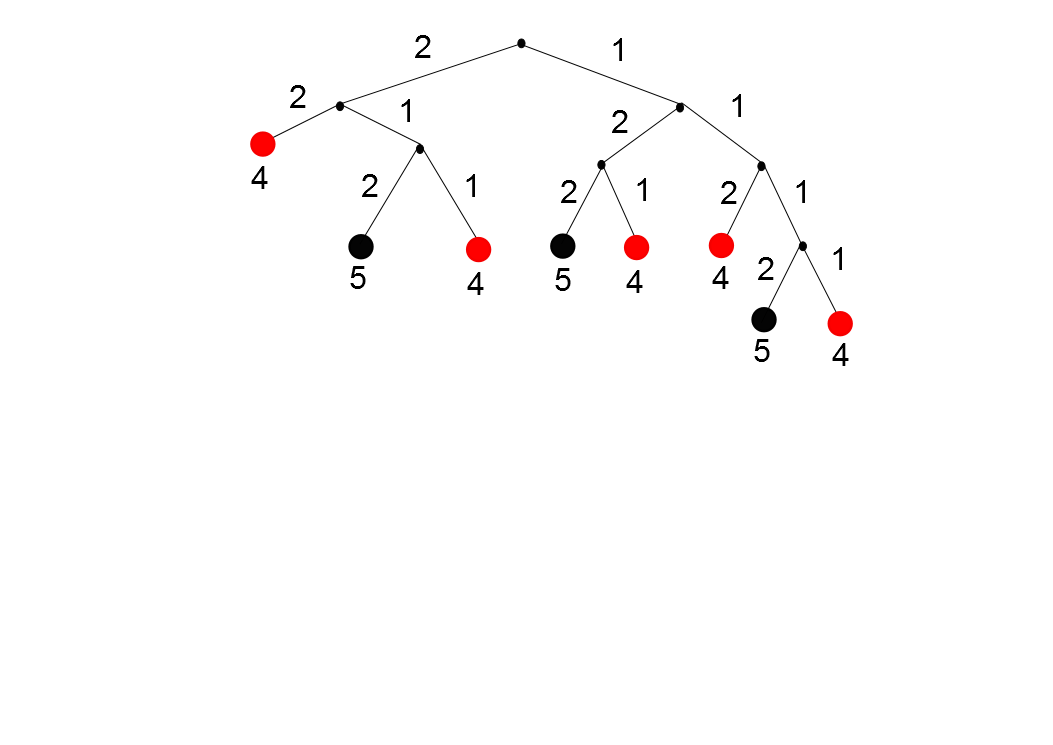}
\vskip -4cm
\caption{The tree structure underlying the set $T(\theta,3)$ for the
golden bee.}
\label{fig:tree}
\end{figure}

An issue is that $T(\theta)$, for some $\theta\in[N]^{\omega}$, may cover just
a ``wedge" --- a closed subset of the plane between two rays, for example a
quadrant of the plane. That this almost never occurs is reflected in
Theorem~\ref{thm:main0} and is discussed in detail in Question 2 of
Section~\ref{sec:ss}.

Our method, encapsulated in Defintion~\ref{def:scaled}, assumes a generating
pair $(p,F)$ but does not find one. Examples appear in
Section~\ref{sec:examples}, and questions concerning their existence appear in
Section~\ref{sec:remarks}.

\begin{theorem}
\label{thm:main0} For any generating pair, there exist infinitely many
$\theta\in\lbrack N]^{\omega}$ for which $T(\theta)$ is a self-similar
polygonal tiling.
\end{theorem}

\section{Examples.} \label{sec:examples}

Viewed geometrically, equation~\eqref{eq:IFS} states that polygon $p$ is a
disjoint union of the smaller similar polygons $f_{1}(p), f_{2}(p), \dots,
f_{N}(p)$. Tilings of polygons by smaller polygons has a long history. For
example, in a 1940 paper, Brooks, Smith, Stone, and Tutte \cite{BSST}
investigated the problem of tiling a rectangle with squares of different
sizes. In 1978 Duijvestijn \cite{D}, by computer, showed that the smallest
possible number of squares in a tiling of a larger square by smaller squares
of different sizes is 21. In general, the term for a tiling of a polygon $p$
by pairwise non-congruent smaller similar copies of $p$ is a \textit{perfect}
tiling. A tiling of a polygon $p$ by smaller similar copies of $p$, all
congruent, is called a \textit{rep-tiling} and $p$ is called a
\textit{rep-tile}. The term was coined by S. W. Golomb \cite{Gol}; also see
\cite{Gardner}. For a generating pair, the smaller similar copies of $p$ need
not be pairwise congruent nor pairwise non-congruent. The term for a polygon
$p$ that is the disjoint union of smaller similar polygons seems to be
\textit{irreptile}; see \cite{R,S,Sch}. For $(p,F)$ to be a generating pair,
the polygon $p$ must be an irreptile that satisfies the ratio condition in equation~\eqref{eq:I}.

Examples~\ref{ex:RT} and \ref{ex:trap} below provide two infinite families of
generating pairs. Example~\ref{ex:spor} provides a few of what we call
sporadic examples. Self-similar polygonal tilings of order $1$ are fairly
common because there are many known rep-tiles. Therefore self-similar
polygonal tilings of higher order, not being prevalent, are illustrated in
this section.
 
\begin{figure}[htb]
\centering
\includegraphics[width=5cm, keepaspectratio]{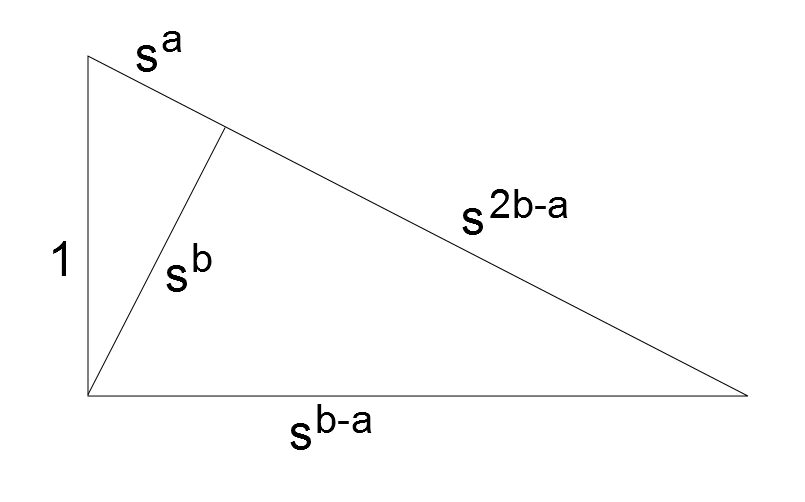} 
\vskip -3mm
\caption{A right triangle in the family $p(a,b)$; see Example~\ref{ex:RT}.}
\label{fig:RT}
\end{figure}

\begin{figure}[htb]
\vskip -3mm 
\centering
\hskip -3mm 
\includegraphics[width=7.5cm, keepaspectratio]{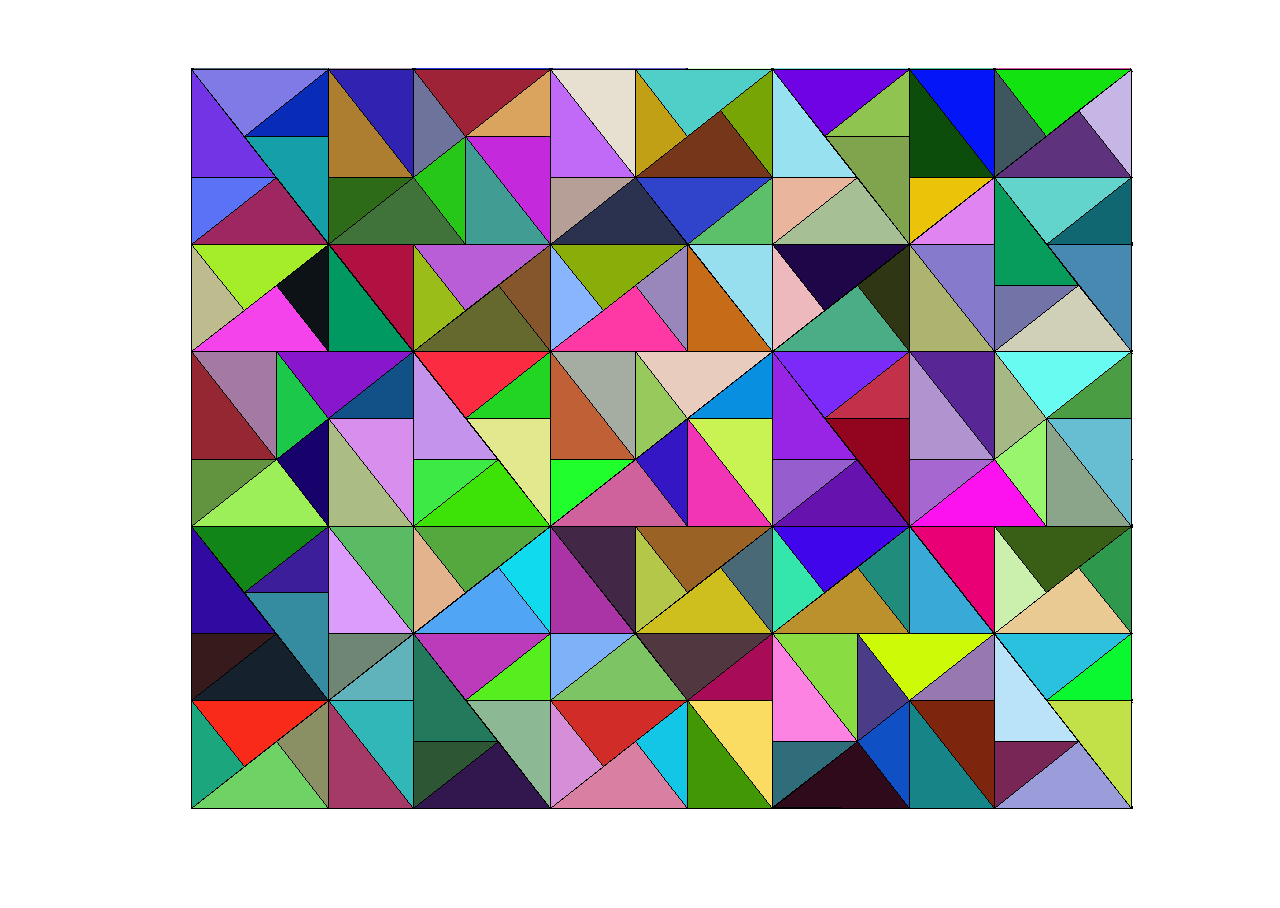}
\hskip -8mm
\includegraphics[width=7.5cm, keepaspectratio]{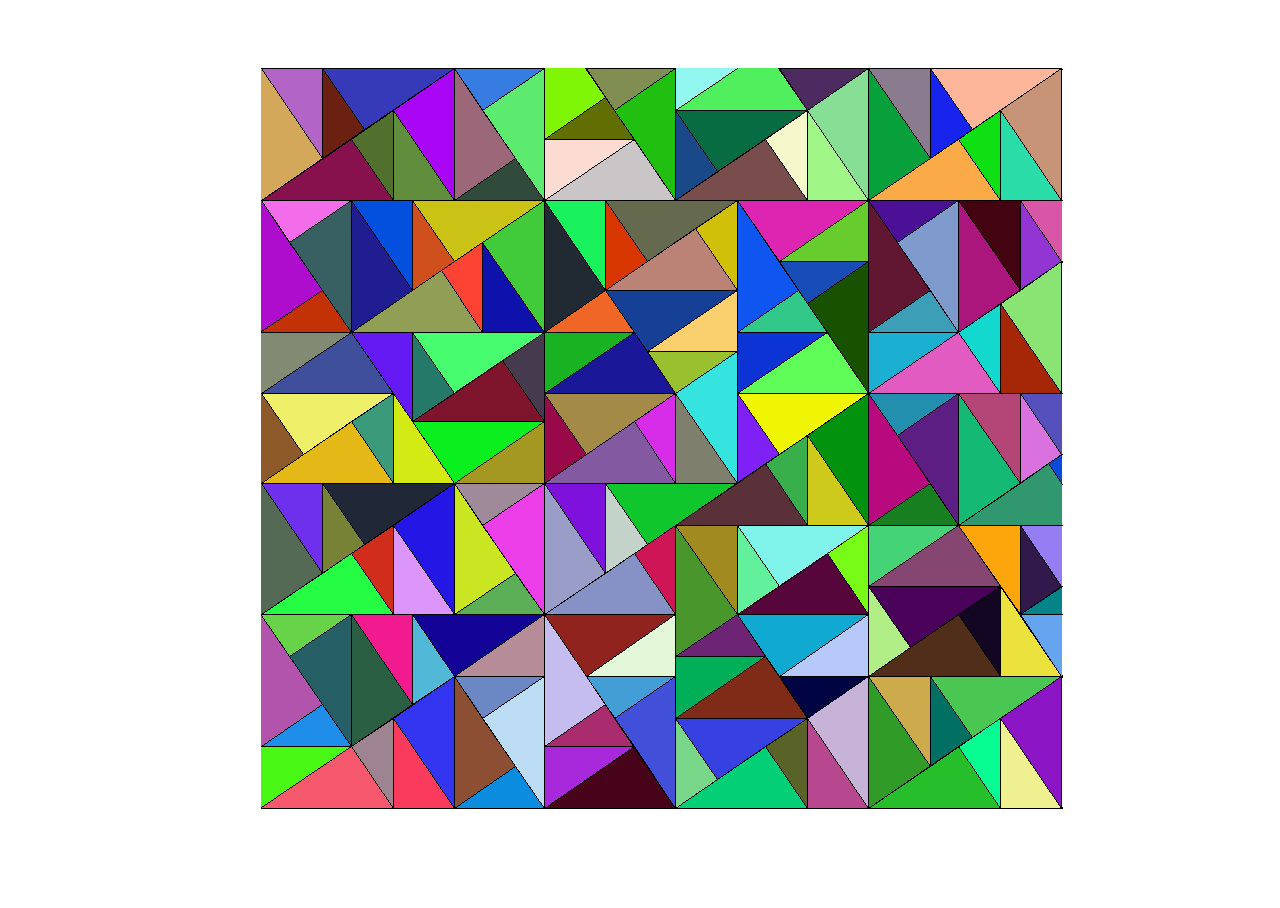}\caption{Two tilings
by right triangles, based on $p(2,1)$ and $p(3,1)$, respectrively; see
Example~\ref{ex:RT}.}
\label{fig:6}
\end{figure}

\begin{example}
[Right Triangles]\label{ex:RT} Consider a right triangle decomposed into two
smaller similar triangles; see Figure~\ref{fig:RT}. For every distinct pair of
positive integers $a,b$, let $s^{2}$ be the unique positive solution of $x^{a}
+ x^{b} = 1$; this is equation~\eqref{eq:ratio}. The triangle $p(a,b)$ in
Figure~\ref{fig:RT} is an irreptile for which the scaling ratios, as given in
\eqref{eq:I}, are:
\[
r_{1} = s^{a}, \qquad\qquad\qquad r_{2} = s^{b}.
\]
If we denote the corresponding set of similitudes by $F(a,b) = \{f_{1}
,f_{2}\}$, then $\left(  p(a,b), F(a,b) \right)  $, where $a > b \geq1$, is an
infinite family of generating pairs. Figure~\ref{fig:6} illustrates two
corresponding self-similar tilings by right triangles, one of order $2$, the
other of order $3$.
\end{example}

\begin{example}
[Trapezoids]\label{ex:trap} Consider a trapezoid decomposed into four smaller
similar trapezoids as in Figure~\ref{fig:4}. The length $w$ has the form $w =
s^{(b-a)/2}$, where $a > b \geq1$ are any two positive integers of the same
parity, and $s$ is the unique solution of $x^{a} + x^{b} = 1$, coming from
equation~\eqref{eq:ratio}: $(x^{a}+x^{b})^{2} = x^{2a} + x^{a+b} +x^{a+b} +
x^{2b} = 1$. The trapezoid, denoted $q(a,b)$, is an irreptile with scaling
ratios:
\[
r_{1} = s^{a}, \qquad\qquad\qquad r_{2}= r_{3} = s^{(a+b)/2}, \qquad
\qquad\qquad r_{4} = s^{b}.
\]
The tiling on the right in Figure~\ref{fig:4} is a self-similar polygonal
tiling based on the case $a=3, b=1$.
\end{example}

\begin{figure}[htb]
\centering
\vskip -6mm
\hskip 5mm \includegraphics[width=.3 \textwidth]{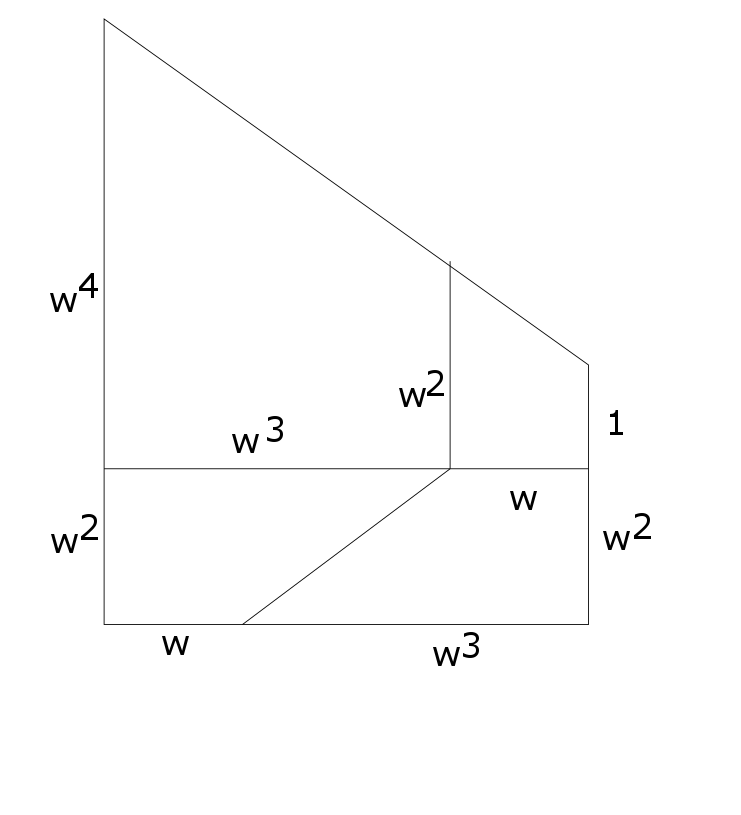} \includegraphics[width=.6 \textwidth]{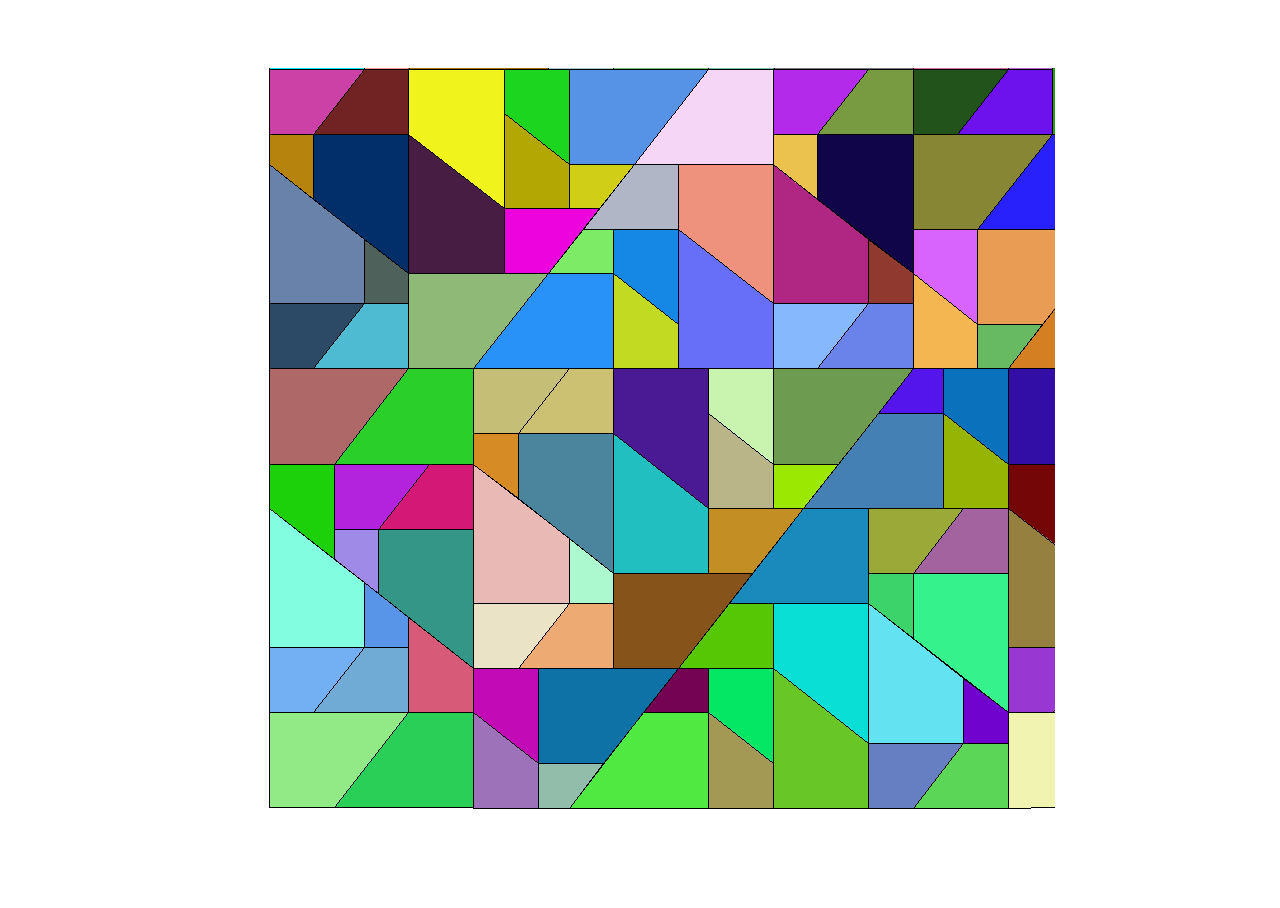} 
\vskip -3mm
\caption{A trapezoid irreptile and a tiling of order $3$ based on
$q(3,1)$; see Example~\ref{ex:trap}.}
\label{fig:4}
\end{figure}

As mentioned in Section~\ref{sec:intro}, the ratio of large to small tiles in
any self-similar golden tiling is, in the limit, the golden ratio. In general,
given a generating pair $(p,F)$, let $\{a_{1},a_{2},\dots,a_{N}\}$ and $s$ be
as in Definition~\ref{def:GP}. Further, let $M=\max\{a_{i}\,:\,i\in\lbrack
N]\}$ and let $D_{R}$ be a disk of radius $R$ centered at the origin. For
$i=1,2,\dots,M$, let

\[
d_{i}=\lim_{R\rightarrow\infty}\frac{\text{the number of tiles congruent to
$s^{i-1}(p)$ in $D_{R}$ }}{\text{the total number of tiles in $D_{R}$}}.
\]
For any golden bee tilings, the proportion of large tiles is $d_{1}
=1/\tau\approx0.6180$ and the proportion of small tiles is $d_{2}
=1-1/\tau\approx0.3820$, where $\tau$ is the golden ratio. For the trapezoid
tiling $Q(3,1)$ of Example~\ref{ex:trap} in Figure~\ref{fig:4}, the
proportions are, $d_{1}\approx0.3826,\;d_{2}\approx0.4392,\;d_{3}
\approx0.1781$, the proportion for the largest of the three tiles being
$d_{1}$, the proportion for the smallest being $d_{3}$. These numbers are
derived as follows. Let

\[C=
\begin{pmatrix}
c_{1} & 1 & 0 & 0 & \cdots & 0\\
c_{2} & 0 & 1 & 0 & \cdots & 0\\
&  & \ddots &  &  & \\
c_{M-1} & 0 & 0 & 0 & \cdots & 1\\
c_{M} & 0 & 0 & 0 & \cdots & 0
\end{pmatrix},
\]
where $c_{i},\,i=1,2,\dots,M,$ is the number of functions in $F$ with
scaling ratio $s^{i}$. Letting $\mathbf{c}=(c_{1},c_{2},\dots,c_{M})$ and
$\mathbf{d}=(d_{1},d_{2},\dots,d_{M})$, it can be shown that,

\[
\mathbf{d}=\frac{C^{n}\ast\mathbf{c}}{\mathbf{j}\ast(C^{n}\ast\mathbf{c})},
\]
where $\mathbf{j}$ is the all ones vector. This holds for any $\theta$-tiling
by the generating pair $(p,F)$, independent of $\theta$.

\begin{figure}[htb]
\centering
\vskip 2mm
\includegraphics[width=12cm, keepaspectratio]{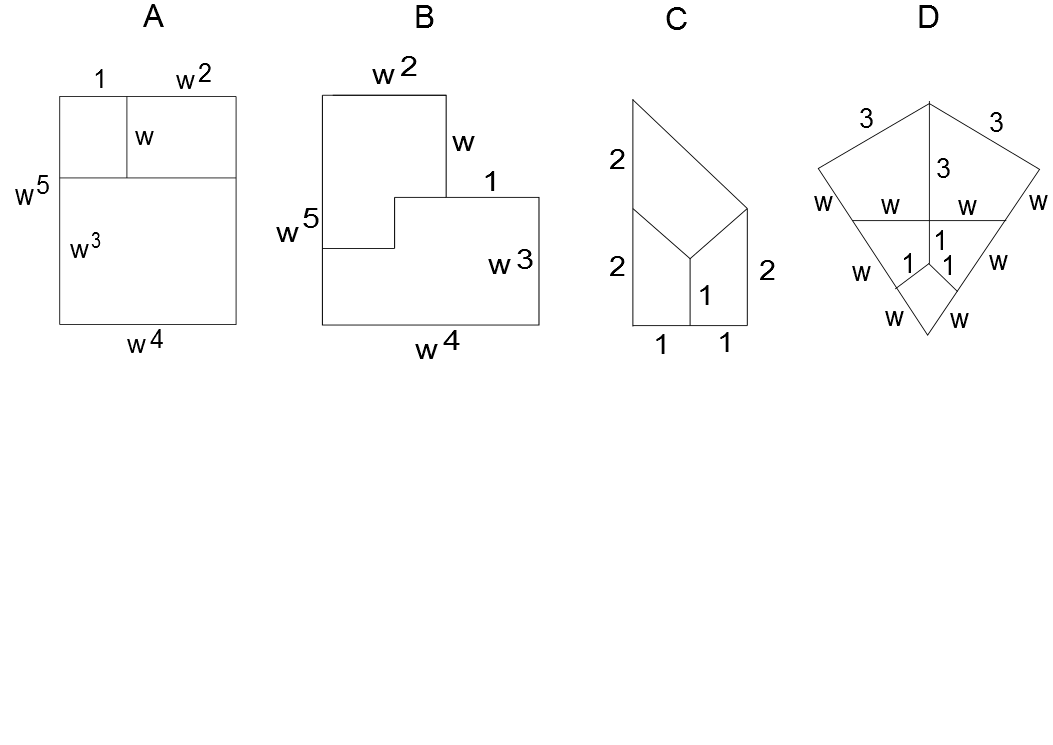}
\vskip -4.2cm
\caption{Sporadic irreptiles; see Example~\ref{ex:spor}.}
\label{fig:spor}
\end{figure}

\begin{example}
[Sporadic generating pairs]\label{ex:spor} The irreptiles in
Figure~\ref{fig:spor} do not belong to an infinite family. For that reason we
call them and the associated generating pairs \textit{sporadic}. For polygons
A and B, the constant $w=\sqrt{\tau}$, where $\tau$ is the golden ratio; in
figure D, the constant is $w=\sqrt{3}$. The scaling ratios are:
\[
\begin{aligned} A: & \quad r_1 = \frac{1}{\sqrt{\tau}}, \; r_2 =
\frac{1}{\tau\sqrt{\tau}}, \; r_3 = \frac{1}{\tau ^2} \\ B: & \quad r_1=
\frac{1}{\sqrt{\tau}},\; r_2 = \frac{1}{\tau} \\ C: & \quad r_1 =
\sqrt{2}/2, \; r_2 = r_3 =1/2\\ D: & \quad r_1 = r_2 = \sqrt{3}/3, \; r_3 =
r_4 = r_5 = 1/3. \end{aligned}
\]
The golden bee in figure B was discussed in Section~\ref{sec:intro}.
Figures~\ref{fig:5}, \ref{fig:2}, \ref{fig:A}, and \ref{fig:8} illustrate self-similar
polygonal tilings based on D, B, A, and C, respectively. Other self-similar
polygonal tilings based on sporadic generating pairs appear in
Figures~\ref{fig:3}, \ref{fig:9} and \ref{fig:10}. And there are many more
sporadic generating pairs.
\end{example}

\begin{example}
[Reducible generating pairs]\label{ex:reducible} Given a generating pair
$(p,F)$, there is a trivial way to obtain infinitely many related generating
pairs. Replace a function $f\in F$ (or several functions) by the set of
functions $\{ f\circ f_{n} \,:\, n=1,2,\dots, N\}$. An example is depicted
geometrically in Figure~\ref{fig:rect}, where one rectangle in the subdivision
of $p$ (right figure) is replaced by three smaller similar rectangles (left
figure). More generally, call a generating pair $(p,F)$ \textit{reducible} if
there is a proper subset S of the tiles in $p$ such that $\bigcup\{t : t \in
S\}$ is similar to $p$; otherwise call $(p,F)$ \textit{irreducible}. A
$\theta$-tiling is called \textit{(ir)reducible} if its generating pair is (ir)reducible.

\begin{figure}[htb]
\centering
\includegraphics[width=6cm, keepaspectratio]{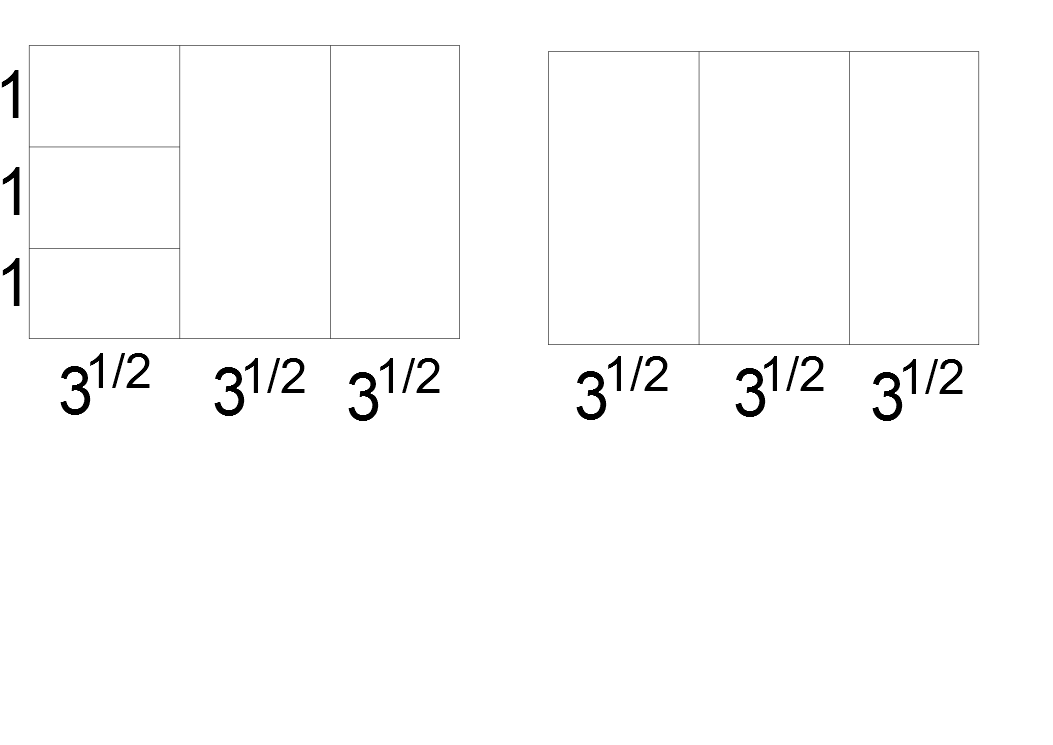}
\vskip -1.5cm
\caption{A reducible irreptile.}
\label{fig:rect}
\end{figure}
\end{example}

\begin{example}
[An irreptile that does not induce a generating pair]\label{ex:nGP} Most
irreducible irreptiles seem not to induce a generating pair; they fail to
satisfy the ratio condition in equation~\eqref{eq:I}. The equilateral triangle
in Figure~\ref{fig:NoGen} is subdivided into 6 smaller similar equilateral
triangles. The two scaling ratios are $1/3$ and $2/3$.
However, the existence of a real number $s$ and integers $a,b$ such that
$s^{a}=1/3$ and $s^{b}=2/3$ would imply that $\log3/\log2$ is rational.

\begin{figure}[tbh]
\vskip -3mm
\centering
\includegraphics[width=4cm, keepaspectratio]{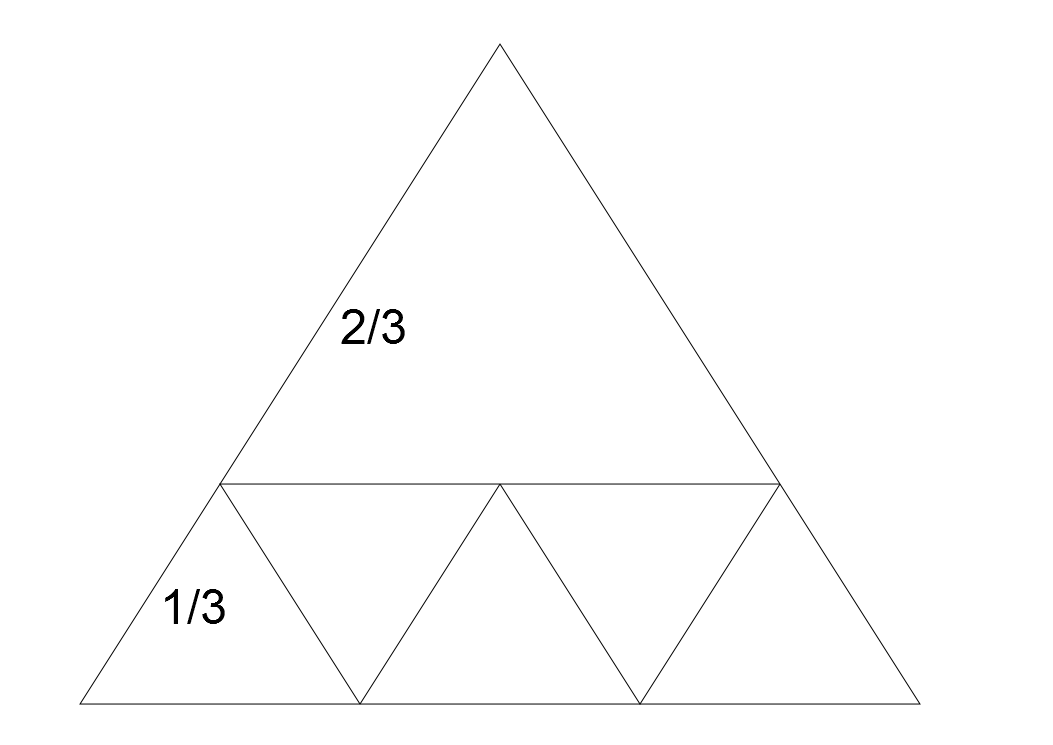} 
\caption{An irreptile that does not induce a generating pair; see Example~\ref{ex:nGP}.}
\label{fig:NoGen}
\end{figure}

\end{example}

\section{The fine points.}  \label{sec:ss}

For a given generating pair $(p,F)$, the following questions concerning
Theorem~\ref{thm:main0} are addressed in this section. Proofs of 
statements appear in Section~\ref{sec:proofs}. A string $\theta\in\lbrack N]^{\omega}$ is
\textit{periodic} if it is a concatenation of the form $\overline{\alpha
}:=\alpha\,\alpha\,\alpha\cdots$, where $\alpha$ is a finite string. A string
$\theta\in\lbrack N]^{\omega}$ is \textit{eventually periodic} if it is a
concatenation of the form $\beta\overline{\alpha}$, where $\alpha$ and $\beta$
are finite strings. 
\vskip 2mm

\noindent\textbf{Questions} \vskip 2mm For which $\theta\in[N]^{\omega}$

\begin{enumerate}
\item are distinct tiles in $T(\theta)$ pairwise disjoint?

\item does $T(\theta)$ fill the plane?

\item is $T(\theta)$ of finite order, and what is the order?

\item is $T(\theta)$ self-similar?

\item is $T(\theta)$ quasiperiodic?
\end{enumerate}

 \noindent Answers are summarized as follows. \vskip 3mm

\textbf{Answer 1}. Distinct tiles in $T(\theta)$ are pairwise disjoint for all
$\theta\in[N]^{\omega}$. This follows from Proposition~\ref{prop:1} in Section~\ref{sec:proofs}.  
\vskip 2mm

\textbf{Answer 2}. It is a tiling of the whole plane for ``almost all"
$\theta\in[N]\in[N]^{\omega}$ in the following three senses.

First, there are infinitely many eventually periodic strings $\theta$ for
which $T(\theta)$ tiles the entire plane. See Proposition~\ref{prop:full} in 
Section~\ref{sec:proofs}.

Second, $T(\theta)$ fills the plane for all disjunctive $\theta$. An infinite
string $\theta$ is \textit{disjunctive} if every finite string is a
consecutive substring of $\theta$. An example is the binary
\textit{Champernowne sequence}
\[
0\,1\,00\,01\,10\,11\,000\,001\cdots,
\]
formed by concatenating all finite binary strings in lexicographic order.
There are infinitely many disjunctive sequences in $[N]^{\omega}$ if $N\geq2$.
See \cite{tilings} for a discussion and proofs of this and the next statement.

Third, define a word $\theta\in\lbrack N]^{\omega}$ to be a \textit{random
word} if there is a $p>0 $ such that each $\theta_{k}$, for $k=1,2,\dots$, is
selected at random from $\{1,2,\dots,N\},$ where the probability that $\theta
_{k}=n$, for $n\in\lbrack N],$ is greater than\ or equal to $p,$ independent
of the preceeding outcomes. If $\theta\in[N]^{\omega}$ is a random word, then,
with probability $1$, the tiling $T(\theta)$ covers ${\mathbb{R}}^{2}$. \vskip 2mm

\textbf{Answer 3}. The tiling $T(\theta)$ is of finite order for all
$\theta\in[N]^{\omega}$. The order, i.e., the number of tiles in the prototile
set, is equal to $M = \max\{ a_{i} : i \in[N]\}$. See
Proposition~\ref{prop:tiling} in Section~\ref{sec:proofs}. 
\vskip 2mm

\textbf{Answer 4}. If $\theta=\overline{\alpha},\,\alpha\in\lbrack N]^{\ast}$
is periodic, then $T(\theta)$ is self-similar with self-similarity map
$\phi=f_{-\alpha}$.

To better understand this answer, note that the set $[N]^{*}$ is a semigroup, where
the operation is concatenation. Let $\mathbb{T}$ denote the set of all
$\theta$-tilings for the pair $(p,F)$. There is a natural semigroup action
\[
{\widehat{\alpha}} : \mathbb{T }\rightarrow\mathbb{T}
\]
of $[N]^{*}$ on $\mathbb{T}$ defined by:
\[
{\widehat{\alpha}} ( T(\theta) ) = T( \alpha\, \theta)
\]
for $\alpha\in[N]^{*}$ and $T(\theta) \in\mathbb{T}$. If $\theta=
\overline\alpha$ is periodic, then clearly $\widehat{\alpha}(T(\theta)) =
T(\theta)$. It can then be shown that any such fixed point $T(\theta)$ of
$\widehat{\alpha}$ is self-similar.

More generally, if $\theta$ is eventually periodic of the form $\theta=
\beta\overline\alpha$ where $e(\beta) = e(\alpha)$, then $T(\theta)$ is
self-similar. This statement follows from the one above for the following
reason. Call two tilings \textit{congruent} if one can be obtained from the
other by a Euclidean motion, i.e., by a translation, rotation, reflection or
glide. Under the given assumptions, tilings $T(\overline\alpha)$ and
$T(\beta\overline\alpha)$ can be shown to be congruent. It is also not hard to
show that if $T(\theta)$ is self-similar and $T(\theta^{\prime})$ is congruent
to $T(\theta)$, then then $T(\theta^{\prime})$ is also self-similar. 

The results above are stated and proved formally in Theorem~\ref{thm:ss} and 
Corollary~\ref{cor:ss} in  Section~\ref{sec:proofs}. 
\vskip 2mm

\textbf{Answer 5}. The tiling $T(\theta)$ is quasiperiodic for all $\theta$
such that $T(\theta)$ tiles ${\mathbb{R}}^{2}$.  See Theorem~\ref{thm:quasi}
in Section~\ref{sec:proofs}. 
\vskip 2mm

In light of the answers to Questions 1-5, there are infinitely many $\theta$ that meet the
requirements for $T(\theta)$ to be a self-similar polygonal tiling, thus
verifying Theorem~\ref{thm:main0}.

\begin{figure}[htb]
\includegraphics [width=12cm, keepaspectratio]{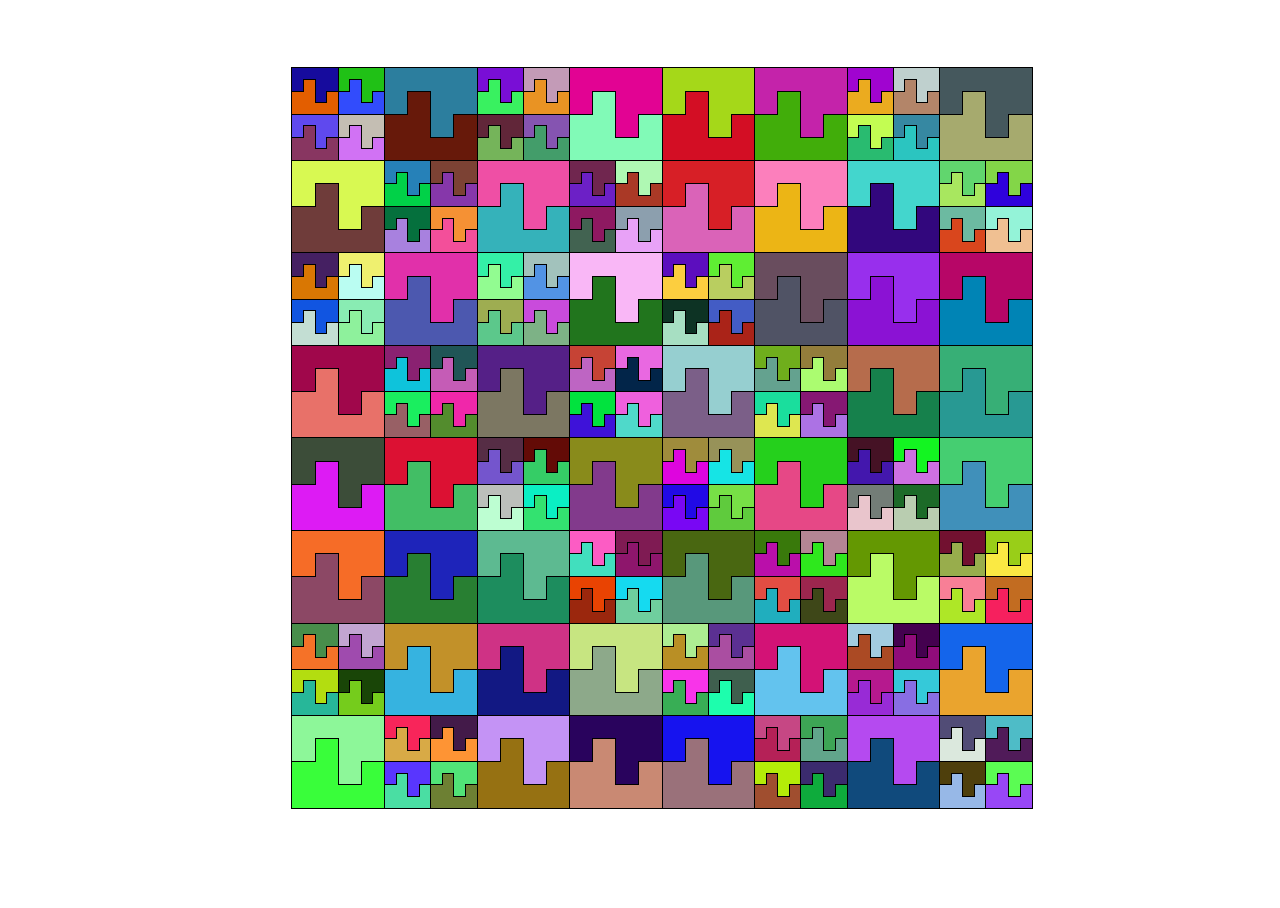} \vskip -4mm\caption{A
self-similar polygonal tiling of order $2$ based on a sporadic generating
pair; see Example~\ref{ex:spor}.}
\label{fig:9}
\end{figure}

\section{Proofs} \label{sec:proofs}

This section contains the proofs of many of the statements in Section~\ref{sec:ss}.

\begin{proposition}  \label{prop:1}
\label{prop:union} For any positive integer $n$, let $S_n = \{ \sigma \in
[N]^* \, : \, e(\sigma) \geq n > e^-(\sigma) \}$. Then 
\begin{equation*}
\bigsqcup \, \{ f_{\sigma}(p) \, : \, \sigma \in S_n \} = p.
\end{equation*}
\end{proposition}

\begin{proof}
The proof is by induction on $n$. If $n=1$, then the equation above is
identical to equation~\eqref{eq:IFS}. Assuming that the equation is true for 
$n$, we will show that it is true for $n+1$. It is routine to check that 
\begin{equation*}
S_{n+1} = \left \{ \sigma \in S_n : \sigma \in S_n \cap S_{n+1} \right \}
\cup \left \{ \sigma \, k : \sigma \in S_n \setminus S_{n+1}, \, k=1,2,
\dots , N\right \}.
\end{equation*}
Now 
\begin{equation*}
\begin{aligned} \bigsqcup &\left \{ f_{\sigma}(p) \, : \, \sigma \in
S_{n+1}\right \} \\
 &= \bigsqcup \left \{f_{\sigma}(p) : \sigma \in S_n\cap S_{n+1} \right
\} \sqcup \bigsqcup \left \{ f_{\sigma\, k}(p) : \sigma \in S_n \setminus S_{n+1}, \,k=1,2
\dots, k \right \} \\ &= \bigsqcup \left \{f_{\sigma}(p) : \sigma \in S_n\cap S_{n+1}
\right \} \sqcup \left \{ f_{\sigma} \left ( \bigsqcup_{k=1}^N f_k(p) \right ) :
\sigma \in S_n \setminus S_{n+1} \right \} \\ &= \bigsqcup \left \{ f_{\sigma}(p) :
\sigma \in S_n\cap S_{n+1} \right \} \sqcup \bigsqcup \left \{ f_{\sigma} (p) : \sigma
\in S_n \setminus S_{n+1} \right \} \\ &= \bigsqcup \left \{ f_{\sigma}(p) : \sigma \in S_n
\right \} = p, \end{aligned}
\end{equation*}
where the second to last equality is by the induction hypothesis. 
\end{proof}

\begin{proposition}
\label{prop:full} Given generating pair $(p,F)$, there are infinitely many
eventually periodic strings $\theta$ such that $T(\theta)$ tiles ${\mathbb{R}%
}^2$.
\end{proposition}

\begin{proof}
It follows from Proposition~\ref{prop:union} that, if $n$ is sufficiently large,
then, for most finite strings $\alpha$ with $|\alpha| \geq n$, the set $%
f_{\alpha}(p)$ is contained in the interior $p^o$ of $p$.  Therefore, the 
unique fixed point of $f_{\alpha}$ lies in  $p^o$.  
With the notation $g^k$ meaning the $k$-wise composition of function
$g$,  $\bigcup_{k=1}^{\infty} (f_{-\alpha})^{k} (p)
 = {\mathbb{R}}^2$  by the Banach contraction mapping theorem.  For any such $\alpha$
and any finite string $\beta$, if $\theta = \beta \, \overline{\alpha}$,
then, 
\begin{equation*}
\bigsqcup \{t \, : \, t \in T(\theta)\} =\bigcup_{k=1}^{\infty}
f_{-(\theta|k)} (p) = f_{-\beta} \left (\bigcup_{k=1}^{\infty}
(f_{-\alpha})^{k} (p) \right )= {\mathbb{R}}^2,
\end{equation*}
the first equality by Proposition~\ref{prop:union} .
\end{proof}

\begin{proposition}
\label{prop:tiling} For any generating pair $(p,F)$ and $\theta\in[N]^{\omega
}$, the prototile set of $T(p,F, \theta)$ consists of $M = \max\{a_{i} \, : \,
i \in[N] \}$ tiles similar to $p$.
\end{proposition}

\begin{proof}
Since the inverse of a similarity and any composition of similarities is a
similarity, each tile $t(\theta,k,\sigma)$ is similar to $p$. Therefore, all
tiles in $T(\theta)$ are similar to $p$.

To show that there are at most finitely many tiles up to congruence,
consider the tiles in $T(\theta,k)$ for any $k\geq 1$. The scaling ratio of a tile in 
$T(\theta,k)$ is of the form
\begin{equation*}
r(f_{-(\theta \, | \, k)} \circ f_{\sigma}) = s^a, \qquad \qquad \text{where}
\qquad \qquad a = e(\sigma) - e(\theta|k).
\end{equation*}
Let $\sigma = \sigma_1\, \sigma_2 \, \cdots \sigma_K$.
The restriction $e(\sigma) \geq e(\theta|k) > e^-(\sigma) = e(\sigma) - a_{\sigma_K}$
immediately implies that
\begin{equation*}
a_{\sigma_K} > a \geq 0.
\end{equation*}
Therefore $a \in \{0,1,2, \dots, M-1\}$, verifying that there are at most $M$ tiles up to congruence. 

To show that $a$ can take any value in $\{0,1,2, \dots, M-1\}$, recall that we can assume that $\gcd(a_{1},\dots,a_{N})=1$ and,  by an elementary result from number theory, every sufficiently large positive
number is a sum of terms $a_{1} ,a_{2},\dots,a_{N}$.  Therefore, if $k$ is sufficiently large, then
$ a = e(\sigma) - e(\theta|k)$ can take any positive integer value subject to the restriction $e(\sigma) \geq e(\theta|k) > e^-(\sigma)$ or equivalently,  subject to the restriction $0\leq a < a_{\sigma_K}$.  But for $k$ sufficiently large,  $\sigma_K \in [N]$ can be chosen arbitrarily, so
that $a_{\sigma_K}$ can be chosen to be $M$.
\end{proof}

\begin{figure}[htb]
\centering
\vskip -4mm
\includegraphics[width=12cm, keepaspectratio]{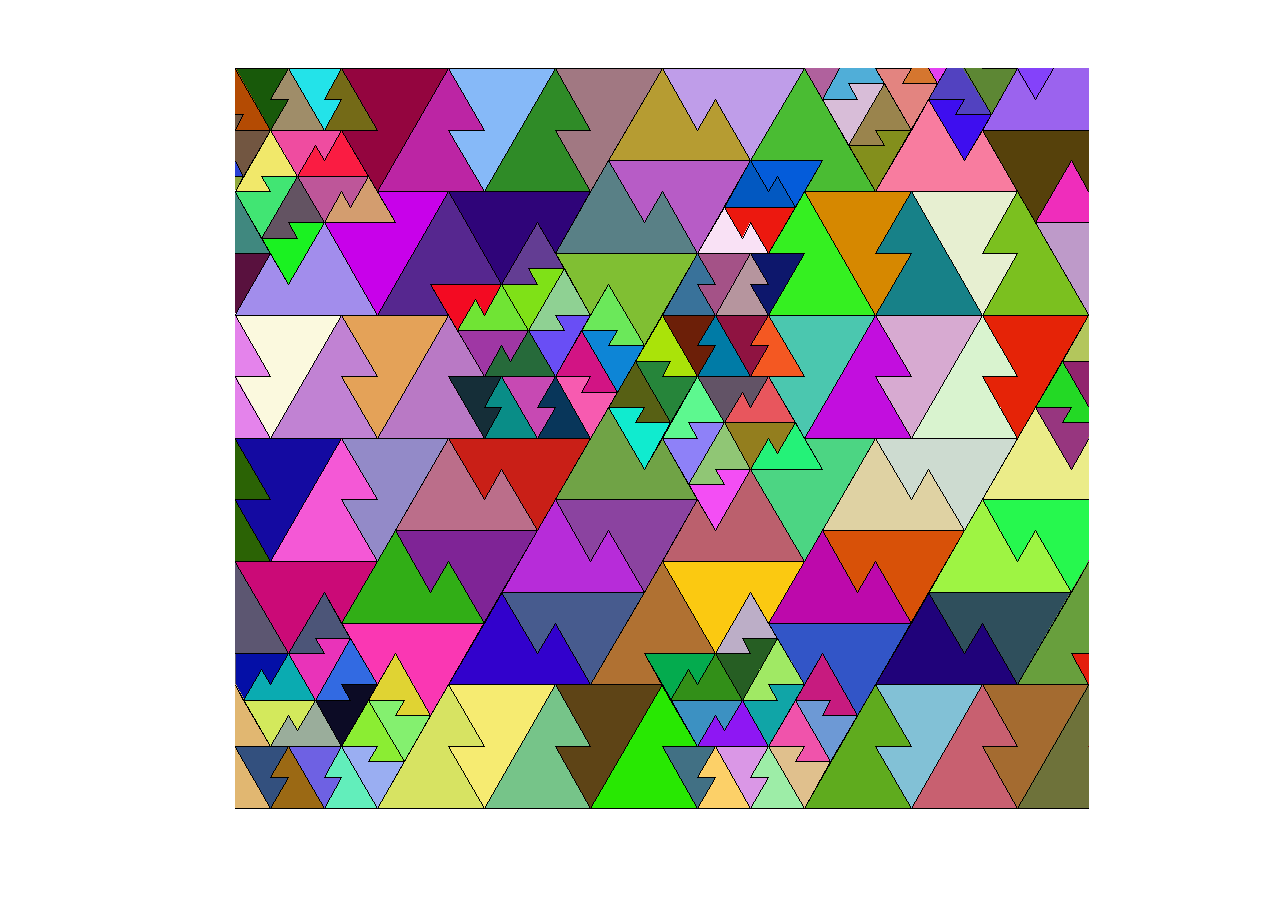}
\vskip -4mm\caption{Another sporadic tiling.}
\label{fig:10}
\end{figure}

\begin{theorem}
\label{thm:ss} If $(p,F)$ is a generating pair and $\theta$ is periodic of
the form $\theta = \overline \alpha$, then $T(\theta)$ is self-similar with
self-similarity map $\phi = f_{-\alpha}$.  
\end{theorem}

\begin{proof}
The set $[N]^*$ is a semigroup, where the operation is concatenation. Let $%
\mathbb{T}$ denote the set of all $\theta$-tilings for the pair $(p,F)$.
There is a natural semigroup action 
\begin{equation*}
{\widehat \alpha}  : \mathbb{T }\rightarrow \mathbb{T}
\end{equation*}
of $[N]^*$ on $\mathbb{T}$ defined by: 
\begin{equation*}
{\widehat \alpha}  ( T(\theta) ) = T( \alpha \, \theta) 
\end{equation*}
for $\alpha \in [N]^*$ and $T(\theta) \in \mathbb{T}$. We claim that, if $%
T(\theta) \in \mathbb{T}$ and if, in the above action, ${\widehat \alpha} $ fixes 
$T(\theta)$, then $\phi := f_{-\alpha}$ is a self-similarity of the tiling $%
T(\theta)$. To prove this, assume that $T := {T}(\theta) \in \mathbb{T}$ and
that ${\widehat \alpha} $ leaves $T$ fixed. It is sufficient to show that, for
every tile $t \in T$ there is a tile $t^{\prime }\in T$ such that $t \subset
\phi(t^{\prime })$. Since ${\widehat \alpha} $ fixes $T$, the tile $t$ can be
written as $t = t(\alpha\theta, k , \sigma)$, where $e(\sigma) \geq
e(\alpha\theta|k) > e^{-}(\sigma)$. Without loss of generality, assume that $k >|\alpha|$. Let 
$k^{\prime }= k-|\alpha|$ and define $m$ to be the least integer such that 
$e(\sigma|m) \geq e(\theta|k^{\prime}\geq e^{-}(\sigma|m).$ With $\sigma_0 = \sigma|m 
$, define $t^{\prime }= t\left (\theta, k^{\prime }, \sigma_0 \right)$. Note
that $f_{\sigma}(p) \subset f_{\sigma_0}(p)$. Now 
\begin{equation*}
\phi(t^{\prime }) = \phi \left ( t(\theta, k^{\prime }, \sigma_0) \right ) =
(f_{-(\alpha \theta)|k }\circ f_{\sigma_0})(p) \supset (f_{-(\alpha
\theta)|k} \circ f_{\sigma})(p) = t(\alpha\theta, k, \sigma) = t.
\end{equation*}
If $\theta = \overline \alpha$ is periodic, then clearly $T(\alpha \theta) = T(\theta)$,
verifying the theorem. 
\end{proof}

Call two tilings {\it congruent} if one can be obtained from the other by a Euclidean
motion, i.e., by a translation, rotation, reflection or glide.  

\begin{lemma} \label{lem:cong}   Let $(p,F)$ be a generating pair.  If $T(\theta)$ is 
self-similar and if $T(\theta')$ is congruent to $T(\theta)$, then 
then $T(\theta')$ is also self-similar.
\end{lemma}

\begin{proof} Let $\phi$ be a Euclidean motion that takes  $T(\theta)$ to $T(\theta')$.
If $f$ is a similarity map for $T(\theta)$, then it is easy to check that 
$\phi \circ f \circ \phi^{-1}$ is a similarity map for  $T(\theta')$. 
\end{proof}

\begin{lemma} \label{lem:ep}
For any generating pair, if $\theta ,\psi \in \lbrack
N]^{\omega }$ share the same ``tail", namely $\theta _{K+1}\theta
_{K+2} \cdots = \psi _{L+1}\psi _{L+2} \cdots$ for a pair of positive integers $K,L
$ such that $e(\theta |K)=e(\psi |L),$ then $T(\theta )$ and $T(\psi )$
are congruent.
\end{lemma}

\begin{proof}  
Let 
\[g = \left ( f_{\psi_{1}}^{-1} \circ f_{\psi _{2}}^{-1} \circ \cdots \circ f_{\psi _{L}}^{-1} \right ) \circ 
\left (f_{\theta_K} \circ f_{\theta_{K-1}} \circ \cdots \circ f_{\theta_1} \right ).\]
Because  $e(\theta |K)=e(\psi |L),$ the transformation $g$ is a Euclidean motion.
Given any tile $t := t(\theta, k, \sigma)$  in $T(\theta)$, it is sufficient to show that $g(t)$
is a tile in $T(\psi)$.  Because the sets $T(\theta,k)$ are nested there is no
loss of generality in assuming that $k>K$.   Let $n = L+k-K > L$ and note that $e(\theta|k) = e(\psi|n)$.
We now have
\[\begin{aligned}
g(t) &=  \left ( f_{\psi_{1}}^{-1} \circ f_{\psi _{2}}^{-1} \circ \cdots \circ f_{\psi _{L}}^{-1} \right ) \circ 
\left ( f_{\theta_K} \circ f_{\theta_{K-1}} \circ \cdots \circ f_{\theta_1} \right ) \circ 
f_{-(\theta|K)} \circ \\ & \qquad \qquad
\left ( f_{\theta_{K+1}}^{-1}f_{\theta _{K+2}}^{-1}...f_{\theta _{k}}^{-1} \right ) \circ
f_{\sigma}(p) \\
&= f_{-(\psi|L)} \circ \left ( f_{\psi_{L+1}}^{-1}f_{\psi _{L+2}}^{-1}...f_{\psi _{L+k-K}}^{-1} \right ) \circ f_{\sigma}(p)
= f_{-(\psi|n)}\circ f_{\sigma}(p) \in T(\psi).
\end{aligned} \]
\end{proof}

\begin{cor} \label{cor:ss}
 If $(p,F)$ is a generating pair and $\theta$ is eventually periodic of
the form $\theta = \beta \overline \alpha$ where $e(\beta) = e(\alpha)$,
then $T(\theta)$ is self-similar.
\end{cor}

\begin{proof} 
 By Lemma~\ref{lem:ep}, the tilings $T(\overline \alpha)$ and
$T(\beta \overline \alpha)$ are congruent.  By Theorem~\ref{thm:ss} the tiling $T(\theta)$ is self-similar.
The result now follows from Lemma~\ref{lem:cong}.
\end{proof}

\begin{theorem}
\label{thm:quasi} 
If $(p,F)$ is a generating pair and $T(\theta)$ tiles 
${\mathbb{R}}^2$, then $T(\theta)$ is quasiperiodic.
\end{theorem}

\begin{proof}
Because it is assumed that $\gcd (a_1, a_2, \dots, a_N) = 1$, there is an integer $a$ such
that, for all $n\geq a$, there is a $\sigma \in [N]^*$ such that $e(\sigma)
= n$. It follows from this, for any $n\geq a$ and $m$ sufficiently large,
that 
\begin{equation*}
Q_m := \{ f_{\sigma}(p) : \sigma \in S_m\},
\end{equation*}
is a refinement of $Q_{m-n}$, refinement meaning that each tile in $Q_m$ is
contained in a tile of $Q_{m-n}$. By Proposition~\ref{prop:union} each $Q_m$ is a
tiling of $p$. Recall that $M = \max \{a_i \, : \, i \in [N] \}$ is the order of the tiling $T(\theta)$, and 
let $b = a+M$.

Let $U$ be a patch in the tiling $T(\theta)$. Since $T(\theta)$ covers ${%
\mathbb{R}}^2$, there is an integer $k$ such that $U$ is a subset of the
union of the tiles in $T(\theta,k)$. Let $c = e(\theta|k)$, and let $R$ be
sufficiently large that a disk of radius $R/2$ contains the polygon $%
s^{b+c}(p)$. If $D$ be any disk of radius $R$, it now suffices to show that $%
D$ contains a set of tiles congruent to $T(\theta|k)$.

Since $T(\theta)$ covers ${\mathbb{R}}^2$, there is a $K$ such that $D
\subset f_{(-\theta|K)}(p)$. Let $C = e(\theta|K)$. For $\sigma \in [N]^*$,
let 
\begin{equation*}
H(\sigma) = \{f_{\sigma} \circ f_{\gamma} (p) : \gamma \in S_c\},
\end{equation*}
and note that, if $e(\sigma) = C-c$, then $H(\sigma) \subset Q_C$.

We next show that, for any $q \in Q_{C-c-b}$, there is a $\sigma$ such that
all tiles in $H(\sigma)$ are contained in $q$. The tile $q = f_{\zeta} (p)$,
for some $\zeta \in S_{C-c-b}$. Note that, by the definition of $Q_{C-c-b}$,
we have $e(\zeta) < C-c-b +M-1$, which implies that $(C-c) - e(\zeta) > b
-M+1 > a$. Therefore there exists a $\sigma^{\prime *}$ such that $\sigma =
\zeta \sigma^{\prime }$ and $e(\sigma) = C-c$. Let $t \in H(\sigma)$. There
is a $\gamma \in S_c$ such that 
\begin{equation*}
t = f_{\sigma} \circ f_{\gamma} (p) = f_{\zeta} \circ f_{\sigma ^{\prime }}
\circ f_{\gamma} (p) \subset f_{\zeta}(p) = q.
\end{equation*}

Let $t = f_{(-\theta|K)}(q)$, where $q \in Q_{C-c-b}$. Then $t=
f_{(-\theta|K)} \circ f_{\zeta}(p)$, where $e(\zeta)\geq C-c-b$. Since $%
e(-\theta|K) + e(\zeta) \geq -c-b$, by the definition of $R$, the tile $t$
is contained in a disk of radius $R/2$. This implies that at least one such
tile $t_0 \in f_{(-\theta|K)} (q_0)$  contained in $D$ because 
\begin{equation*}
D \subset f_{(-\theta|K)}(p) = f_{(-\theta|K)} \left ( \bigcup \{ t : t \in
Q_{C-c-b})\} \right ) = \bigcup \{t: t\in f_{(-\theta|K)} (Q_{C-c-b})\}.
\end{equation*}

We have shown that there is a $\sigma$ with $e(\sigma) = C-c$ such that all
the tiles in $H(\sigma)$ are contained in $q_0$. Therefore, all the tiles in $%
f_{(-\theta|K)} (H(\sigma) )$ are contained in $t_0$, hence also in $D$.
Because $f_{(-\theta|K)} (H(\sigma) )$ is congruent to $T(\theta,k)$, the
proof is complete.
\end{proof}

\section{Open problems.} \label{sec:remarks}

Some basic questions remain open.

\begin{question}
Can every self-similar polygonal tiling be obtained by the generating pair
method of Definition~\ref{def:scaled}?
\end{question}

\begin{question}
In Section~\ref{sec:examples}, Examples 2 and 3, two infinite families of
irreducible generating pairs are given. Are there additional infinite families
of irreducible generating pairs?
\end{question}

\begin{question}
Several sporadic irreducible generating pairs are given in
Section~\ref{sec:examples}. Are there at most finitely many sporadic
irreducible generating pairs? If not, given $N$, are there at most finitely
many irreducible generating pairs $(p,F)$ for which $|F| = N$?
\end{question}

\begin{question}
The pinwheel tilings of C. Radin \cite{Rad} are based on the subdivision of a
right triangle with side lengths $1,2, \sqrt{5}$ due to J. Conway; see
Figure~\ref{fig:pin}. These tilings are order $1$ tilings in the terminology
of this paper. Do there exist higher order analogs? In other words, does there
exist an irreducible (in the sense of Example~\ref{ex:reducible}) self-similar
polygonal tiling of order at least $2$ for which the tiles appear in
infinitely many rotational orientations?
\end{question}

\begin{figure}[htb]
\centering
\vskip -6mm
\includegraphics[width=6cm, keepaspectratio]{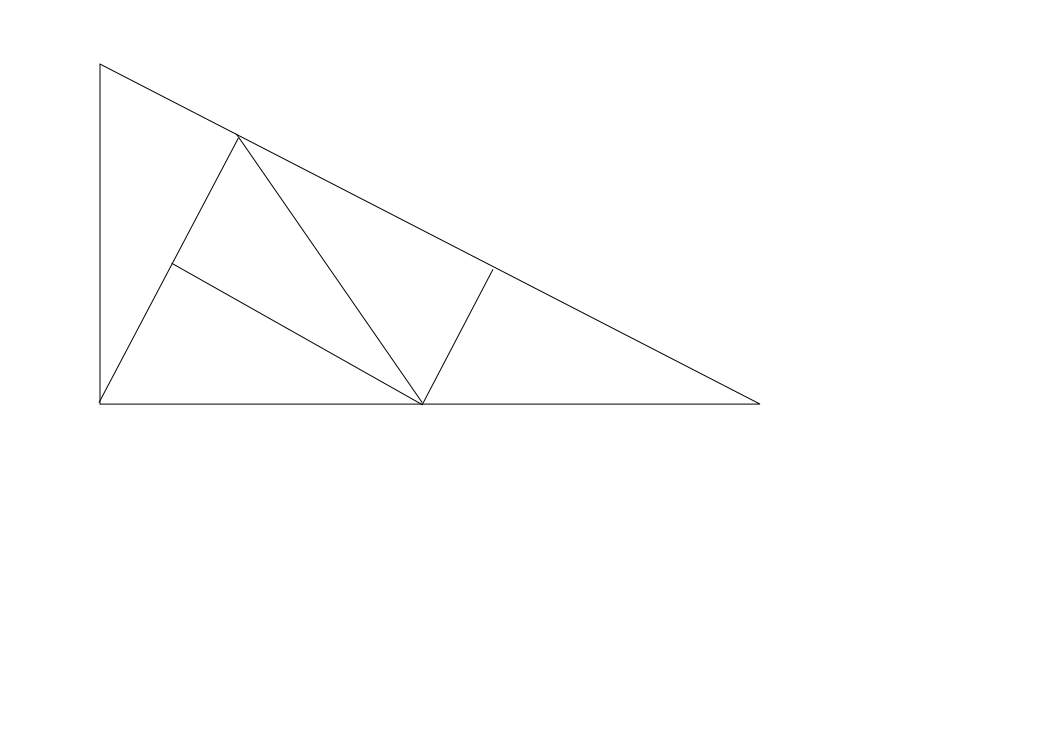}
\vskip -1.6cm
\caption{Pinwheel rep-tile.}
\label{fig:pin}
\end{figure}

Given a generating pair $(p,F)$, let $\mathbb{T}(p,F)$ denote the set of its
$\theta$-tilings of the plane. Let $\pi: [N]^{\omega} \rightarrow
\mathbb{T}(p,F)$ denote the map defined by $\theta\mapsto T(\theta)$. As
stated in Section~\ref{sec:intro}, there are infinitely many self-similar
golden bee tilings up to congruence, none of which is periodic. On the other
hand, the image of $[N]^{\omega}$ under $\pi$ need not always be infinite.
There may be many strings $\theta$ for which their images $T(\theta)$ are
pairwise congruent tilings. This is the case, for example, when $p$ is a
square whose images under four functions in $F$ subdivide $p$ into four
smaller squares. In this case, $T(\theta)$ is, for all $\theta$, the standard
tiling of the plane by squares. Moreover, all such square $\theta$-tilings are periodic.

\begin{question}
Let $(p,F)$ be an irreducible generating pair, where $p$ is not a triangle or
a parallelogram. Is it the case that there exist infinitely many $\theta
$-tilings up to congruence, none of which is periodic?
\end{question}

{\bf Acknowledgment.}
We thank Louisa Barnsley for assistance with some of the figures.
This work was partially supported by a grant from the Simons Foundation (\#322515 to Andrew Vince). 
It was also partially supported by a grant from the Australian Research Council to Michael Barnsley (DP130101738).

\end{document}